\documentclass[12pt]{article}
\usepackage{amsmath,amssymb,amsthm,amsfonts,amscd}
\usepackage{hyperref}
\usepackage[numeric]{amsrefs}
\usepackage[dvips]{graphicx}

\newtheorem{thm}[equation]{Theorem}

\newtheorem{lem}[equation]{Lemma}
\newtheorem{rem}[equation]{Remark}

\newtheorem{prop}[equation]{Proposition}

\DeclareMathOperator{\Ad}{Ad}

\numberwithin{equation}{section}

\renewcommand\a{\alpha}

\newcommand\fa{\mathfrak a}
\newcommand\fb{\mathfrak b}

\newcommand\fg{\mathfrak g}

\newcommand\fk{\mathfrak k}

\newcommand\fm{\mathfrak m}
\newcommand\fn{\mathfrak n}

\newcommand\G{\Gamma}

\newcommand{\N}{{\mathbb{N}}}
\newcommand{\Z}{{\mathbb{Z}}}
\newcommand{\R}{{\mathbb{R}}}

\newcommand{\C}{{\mathbb{C}}}

\newcommand{\Q}{{\mathbb{Q}}}

\newcommand{\Hom}{\operatorname{Hom}}

\renewcommand\({\left(}
\renewcommand\){\right)}

\newcommand{\gobble}[1]{}
  \newcommand{\rangeref}[2]{%
    \ref{#1}--\afterassignment\gobble\fam 0\ref{#2}%
  }

\begin{document}

\title{On the rapid decay of cuspidal automorphic forms}
\author{Stephen D. Miller\thanks{Partially supported
by NSF grant  DMS-0901594} \ and Wilfried Schmid\thanks{Partially supported by DARPA grant HR0011-04-1-0031 and NSF
grant DMS-0500922}}
\date{May 15, 2011}
\maketitle

\begin{abstract}

Many important analytic statements about automorphic forms, such as the analytic continuation of certain
$L$-functions, rely on the well-known rapid decay of $K$-finite cusp forms on Siegel sets.  We extend this here to
prove a more general decay statement along sets much larger than Siegel sets, and furthermore state and prove the
decay for smooth but not necessarily $K$-finite cusp forms.  We also state a general theorem about the convergence
of Rankin-Selberg integrals involving unipotent periods, closing a gap in the literature on $L$-functions.
 These properties serve  as the analytic basis \cite{pairings} of a new method to establish holomorphic
continuations of Langlands $L$-functions, in particular the exterior square $L$-functions \cite{extsq} on $GL(n)$.
\vspace{12pt}

\noindent {\bf Keywords:} Automorphic forms, rapid decay, cusp forms, $L$-functions, Rankin-Selberg, integral
representations, uniform moderate growth.
\end{abstract}

\section{Introduction}\label{sec:intro}

Cuspidal automorphic forms on symmetric spaces decay rapidly on Siegel sets, a statement that is generally
attributed to Gelfand and Piatetski-Shapiro. This fact was later generalized to so-called $K$-finite cuspidal
auto\-morphic forms~-- i.e., cuspidal automorphic forms on quotients $\Gamma\backslash G$ that transform finitely
under right translation by a maximal compact subgroup $K$ of the ambient reductive group $G$ \cite{godement:1966}.

The rapid decay is crucially used to prove the analytic continuation and functional equations of $L$-functions, by
the approach known as ``integral representations", which goes back to Hecke in the case of $SL(2,\R)$. In this
approach, the $L$-function $L(s)$ in question is expressed as an integral of an automorphic form against a suitable
integral kernel, depending on the complex variable $s$; the analytic continuation and functional equation for
$L(s)$ are then deduced from the corresponding properties of the integral kernel. In a number of cases, the
convergence of the integrals depends on the rapid decay on sets larger than Siegel sets. The exterior square
$L$-function for $GL(n)$ \cite{js2} is the first instance of this phenomenon; Jacquet-Shalika carefully establish
the rapid decay on the relevant type of set. Other authors, in later papers, refer to the Jacquet-Shalika estimate
and suggest that similar arguments apply to their cases as well, without supplying any details. In some instances,
the Jacquet-Shalika argument does apply, but in others it does not. In any case, we know of no general argument in
the literature that treats all these cases by a uniform method. One purpose of our paper is to provide such a
general, uniform statement. In fact, our results simultaneously handle the convergence of all integral
representations of $L$-functions that we are aware of. Thus our paper closes a notable gap in the $L$-function
literature.

One other issue is the class of cuspidal automorphic forms that decay rapidly. For our own work \cite{pairings} we
need to know the rapid decay of smooth, but not necessarily $K$-finite, automorphic forms. One can ask for two
levels of generality: smooth cuspidal automorphic forms of uniform moderate growth, or even more generally, smooth
cuspidal automorphic forms that are not assumed to have uniform moderate growth; of course they must be required to
have moderate growth, as is almost always assumed in the context of automorphic forms\begin{footnote}{The standard
arguments for proving the rapid decay of cusp forms use uniform moderate growth in an essential way. For $K$-finite
automorphic forms, the distinction between moderate growth and uniform moderate growth disappears, since the former
implies the latter quite directly in the $K$-finite case.}\end{footnote}. The case of uniform moderate growth is
well understood by experts. Yet much to our surprise, we were unable to find a clear, simple proof~-- or even a
clear, simple statement~-- in print. That is the degree of generality which we need for our own work. However, it
should be of interest that the hypothesis of merely moderate growth suffices. Using a result of Averbuch
\cite{Averbuch:1986}, we shall show that for smooth cuspidal automorphic forms, moderate growth implies uniform
moderate growth, and hence rapid decay.

Let us state our results formally. We suppose that $G$ is the group of real points of a reductive algebraic group
defined over $\Q$; this hypothesis will be relaxed slightly in the main body of the paper. To simplify our
statements, we suppose that $G$ has compact center~-- the general situation can easily be reduced to that case. We
consider a smooth, but not necessarily $K$-finite, automorphic form $F$. We require $F$ to have moderate growth, as
is customary, but do not expressly assume that $F$ has uniform moderate growth. We fix a linear norm $\|\ \|$ on
$G$, i.e., the operator norm with respect to a faithful finite dimensional representation of $G$; the particular
choice of the norm does not matter. We choose $A\subset G$, the identity component of a maximal diagonalizable
$\Q$-split subgroup. By a Siegel set $\mathfrak S\subset G$, we mean one defined in terms of $A$ and some maximal
compact subgroup $K\subset G$ as in \cite{Borel:1969b}, for example.\newline

\noindent{\bf Theorem A.\ } If the automorphic form $F$ is cuspidal, then for every Siegel set $\mathfrak S \subset
G$ and every $n\in\N$, there exists a positive constant $C$ such that\vspace{-6pt}
\[
g\in \mathfrak S\ \ \ \ \Longrightarrow\ \ \ \ |F(g)| \ \leq \ C\,\|g\|^{-n} \,.\vspace{6pt}
\]

In the context of $K$-finite automorphic forms, this is the traditional way of stating the rapid decay of cusp
forms. Siegel sets are well adapted to the analysis of $K$-finite cusp forms, but less so to the setting of merely
smooth cuspidal automorphic forms. Our next statement, without reference to Siegel sets, formally implies theorem
A; conversely, it follows from the proof~-- but not the statement~-- of theorem A. We choose a minimal
$\Q$-parabolic subgroup $P\subset G$ which contains the group $A$, with Langlands-Levi decomposition
\begin{equation}
\label{IMLL}
P \ = \ M \! \cdot\! A\! \cdot \! N \,.
\end{equation}
Thus $M\subset G$ is maximal among reductive $\Q$-subgroups commuting with $A$; $M$ is then necessarily anisotropic
over $\Q\,$.\newline

\noindent{\bf Theorem B.\ } If the automorphic form $F$ is cuspidal, then for every compact subset $S_r \subset G$
and every $n\in\N$, there exists a positive constant $C$ such that\vspace{-6pt}
\[
m\in M\,,\,\ a \in A\,,\,\ g\in S_r\ \ \ \ \Longrightarrow\ \ \ \ |F(mag)| \ \leq \ C\,\|a\|^{-n} \,.\vspace{6pt}
\]

The next result involves a $\Q$-parabolic subgroup $P_1\subset G$ which contains $P$, with Langlands-Levi
decomposition
\begin{equation}
\label{ILL}
P_1 \ = \ M_1 \! \cdot\,\! A_1\! \cdot \,\! N_1
\end{equation}
adapted to that of $P$~-- in other words, $A_1\subset A$, $M_1\supset M$, $N_1\subset N$. We should mention that
any $\Q$-parabolic subgroup of $G$ is conjugate, over $\Q\,$, to one of this type.\newline

\noindent{\bf Theorem C.\ } If the automorphic form $F$ is cuspidal, then for every compact subset $S_r \subset G$
and every $n\in\N$, there exists a positive constant $C$ such that\vspace{-6pt}
\[
m \!\in\! M\,,\, n_1 \!\in\! N_1\,,\, a\! \in\! M_1 \!\cap\! A\,,\, g \!\in\! S_r \,
\ \ \Longrightarrow\, \ \ |F(m\, n_1 \,a\,g)| \, \leq \, C \|a\|^{-n} \,.\vspace{6pt}
\]

In section \ref{sec:unipotent} we give a typical example of an application of this result, an example to which
theorems A or B are not directly applicable. Some applications of theorem C can also be deduced from:\newline

\noindent{\bf Theorem D.\ } Let $\delta_1 : P_1 \to \R_{>0}$ denote the modular function for the quotient $G/P_1$.
If the automorphic form $F$ is cuspidal, then for every compact subset $S_r \subset G$ and every $n\in\N$, there
exists a positive constant $C$ such that\vspace{-6pt}
\[
p \in P_1\,,\ g\in S_r \
\ \ \Longrightarrow\ \ \ |F(p\,g)| \, \leq \, C\left(\delta_1(p)\right)^{-n} \,.\vspace{6pt}
\]

In the $K$-finite context, for maximal parabolic subgroups $P_1\subset G$, this bound is stated in \cite{fj}, but
used only in the case of $G = GL(n,\R)$. For the proof, Friedberg-Jacquet refer the reader to \cite{jr}, which
establishes the bound for maximal parabolic subgroups of $GL(n,\R)$, but not for maximal parabolic subgroups of a
general group $G$.

Both theorems C and D contain theorem B as a special instance, in the case of theorem C with $P_1= G$, and in the
case of theorem D with $P_1=P$; see remark \ref{rem:decay3}. For some applications of the theorems, it is necessary
to understand how the bounding constants $C$ in the four theorems depend on the automorphic form $F$. The precise
description of the dependence, which involves representation theoretic notions, will be given in the main body of
this paper.

Theorem A, for $K$-finite automorphic forms, is often stated without the hypothesis of cuspidality, though bounding
not $F$ itself, but rather the difference between $F$ and a finite linear combination of constant terms along
various cusps; see \cite[section~1.2]{MW}, for example. That type of argument applies quite directly in the smooth
case, assuming uniform moderate growth. Theorems B, C, and D are deduced from Theorem A using reduction theory, so
analogues of these for noncuspidal forms for smooth automorphic forms of uniform moderate growth can be deduced in
a straightforward way.

The estimates we prove in this paper are phrased in classical terms. However, they apply directly to adelic
automorphic forms using the translation between adelic and classical Siegel domains described in
\cite{godement:1962}. As Borel \cite{Borel:1969b} has pointed out, the situation is in fact simpler in the adelic
setting because there is only a single adelic cusp. Alternatively one can exhibit the action of the adeles in terms
of the automorphic distributions we study; see \cite{mirabolic}. From that point of view, the automorphic
distribution itself is the adelic object, and the results in sections \ref{sec:rapiddecay} - \ref{sec:maintheorems}
below provide all the decay estimates we need.

We conclude this introduction with a flow chart for our paper. We recall the notion of an automorphic distribution
and its connection to automorphic forms in section \ref{sec:rapiddecay}, which ends with the statement of two
results: theorem \ref{thm:uniformlymg} asserts that in the context of smooth cuspidal automorphic forms, moderate
growth ensures uniform moderate growth; theorem \ref{thm:decay1} states the rapid decay on Siegel sets for
automorphic forms arising from cuspidal automorphic distributions. The former also implies that all smooth cuspidal
automorphic forms do arise from cuspidal automorphic distributions. Theorem A follows from these two theorems,
taken together. In section \ref{sec:maintheorems} we state versions of theorems B\,--\,D in terms of cuspidal
automorphic distributions, and deduce those versions from theorem \ref{thm:decay1}. As in the case of theorem A,
theorems B\,--\,D follow from the corresponding statements in terms of automorphic distributions. Section
\ref{sec:unipotent} describes the application of theorem C to Rankin-Selberg type integrals which involve unipotent
integrations, and contains a typical example. The proofs of theorem \ref{thm:decay1} and theorem
\ref{thm:uniformlymg} occupy, respectively, sections \ref{sec:proofmainthm} and \ref{sec:uniform}. The appendix
explains certain technical details which are well known to representation theorists, but have not been concisely
described elsewhere.

We gratefully acknowledge helpful discussions with Bill Casselman, Solo\-mon Friedberg, Howard Garland, David
Ginzburg, Herv\'{e} Jacquet, Gregory Margulis, Ze'ev Rudnick, David Soudry, Nolan Wallach, and also with Joseph
Hundley, who read our manuscript carefully.

\section{Rapid decay on Siegel sets}\label{sec:rapiddecay}

We start with a reductive linear group defined over $\Q$, realized as subgroup of $GL(N,\C)$ for some $N$,
compatibly with the $\Q$-structure determined by $\Q^N \subset \C^N$. We let $G$ denote a finite cover, typically
but not always the trivial cover, of the group of real points in that reductive linear group. Following common
practice, we call a subgroup $\G \subset G$ {\it arithmetic} if it is commensurate with a stabilizer of a lattice
in $\,\R^N$\!,\,\ a lattice whose $\Q$-linear span is $\Q^N$. With slight abuse of terminology, we call $g \in G$
{\it rational} if it preserves $\Q^N \subset \C^N$\!, and $G_\Q$ denotes the group of all such rational $g$. We
shall say that a subgroup of $G$ is ``defined over $\Q$" if it is the inverse image, in $G$, of the group of real
points of a $\Q$-subgroup of the original linear group. Analogously we shall call a homomorphism $\phi : G \to H$,
from $G$ to the group $H$ of real points in some other $\Q$-group, ``defined over $\Q$" if it drops to a
homomorphism from the linear quotient of $G$ to $H$, and is defined over $\Q$ on that level.

Going to a further finite cover, if necessary, one can express $G$ as the direct product of a maximal central,
connected, $\Q$-split subgroup and a reductive group whose center is compact and which satisfies the original
hypotheses; see, for example, \cite[Proposition 10.7]{Borel:1969b}. This allows us to impose the following standing
assumption, without loss of generality:
\begin{equation}
\label{cptctr}
Z_G\,,\, \text{the center of $G$,\, is compact.}
\end{equation}
Indeed, any automorphic form on which the center of $G$ acts according to a character is completely determined by
its restriction to the derived group $[G,G]$. Assertions about the rapid decay do not really involve the center.
Thus (\ref{cptctr}) is not a restrictive assumption, but simplifies various definitions and statements below.

We pick a minimal one among the parabolic subgroups, defined over $\Q$, of the original reductive linear group, and
let $P$ denote the corres\-ponding subgroup of $G$. It is unique up to conjugation by some $g \in G_\Q\,$, and
\begin{equation}
\begin{aligned}
\label{pmin}
&P \ = \ M\!\cdot\! A \!\cdot\! N\,,\ \ \text{with}
\\
&\ \ \ A\ \ \text{abelian, connected, $\Q$-split, central in $M\!\cdot \! A$\,,}
\\
&\ \ \ \ \ \ N\ \ \text{unipotent, connected, defined over $\Q$, normalized by $M\!\cdot \!A$\,,}
\\
&\ \ \ \ \ \ \ \ \ \text {and}\ \ M\ \ \text{reductive, defined over $\Q$, anisotropic over $\Q$\,.}
\end{aligned}
\end{equation}
Here ``connected" means connected in the Hausdorff topology. Since we allow finite covers of algebraic groups, we
deviate slightly from Borel's conventions in \cite{Borel:1969b}, which would let $A$ denote the Zariski closure of
what we denote by $A$. In the case of $G = SL(2,\R)$, for example, equipped with the standard $\,\Q$-structure, the
upper triangular subgroup can play the role of $P$; for Borel, $M$ reduces to the identity and $A$ is the entire
diagonal subgroup, whereas for us $M=\{\pm 1_{2\times 2}\}$ and $A$ is the subgroup of diagonal matrices with
positive diagonal entries.

Once and for all we choose a maximal compact subgroup $K \subset G$, whose Lie algebra is perpendicular to the Lie
algebra of $A$, perpendicular relative to the trace form of the tautological representation of $GL(N,\R)$ on
$\C^N$~-- recall that $G$ finitely covers a subgroup of $GL(N,\R)$, so the Lie algebra of $GL(N,\R)$ contains that
of $G$. The assumption (\ref{cptctr}) implies $Z_G \subset K$, of course.

With our conventions, passing to a finite cover of $G$ only affects $M$; both $A$ and $N$ remain connected, $N$ is
still an algebraic group, and $M$ a finite cover of an algebraic group. The Lie algebras of these groups will be
referred to by the corresponding lower case German letters. A restricted root is a non-zero element $\alpha \in
\fa^*$ such that the $\alpha$-{\it root space}
\begin{equation}
\label{rootspace}
\fg^\alpha \ = \ \{\, X\in \fg \ \mid \ [H,X] = \langle \alpha , H \rangle X\ \text{for all $H \in \fa$\,} \}
\end{equation}
is non-zero. Then
\begin{equation}
\label{setofroots}
\Phi(\fa,\fg) \ = \ \{\, \alpha\in \fa^* \ \mid \ \fg^\alpha \neq 0 \,,\,\ \alpha \neq 0 \,\}\,,
\end{equation}
the set of rational roots, is an abstract root system, not necessarily reduced \cite{Borel:1969b}. It contains
\begin{equation}
\label{positiveroots}
\Phi^+(\fa,\fg) \ = \ \{\, \alpha\in \Phi(\fa,\fg) \ \mid \ \fg^\alpha \subset \fn \ \}
\end{equation}
as a positive root system -- in particular, $\Phi(\fa,\fg) = \Phi^+(\fa,\fg)\cup (-\Phi^+(\fa,\fg))$, as disjoint
union. Every $\alpha \in \Phi(\fa,\fg)$ lifts to a character
\begin{equation}
\label{eofroot}
e^\alpha \ : \ A \longrightarrow \R_{>0} \,,\ \ \ \ \ e^\alpha (\exp H) = e^{\langle \alpha , H \rangle}\ \ \text{for $H \in \fa$\,},
\end{equation}
and in fact to an algebraic character of the Zariski closure of $A$. Equivalently, one may view $e^\alpha$ as an
algebraic character of $M\! \cdot \! A$, which is identically equal to 1 on the identity component of $M$ (identity
component in the Zariski sense), but may take the value $\,-1\,$ on some of the other components of $M$.

Statements about the decay of cusp forms on symmetric spaces are traditionally stated in terms of Siegel sets,
which are inverse images, under the projection $G\to G/K$, of approximate fundamental domains of the action of
$\Gamma$ on $G/K$. In dealing with smooth, but not necessarily $K$-finite functions on $G$, growth estimates are
relatively insensitive to right translation by arbitrary compact subsets of $G$, not just $K$ itself. For that
reason it is convenient to introduce the notion of a {\it generalized Siegel set\/}:~a set $\,\mathfrak S =
\mathfrak S(S_\ell, \epsilon,S_r)\,$ of the form
\begin{equation}
\begin{aligned}
\label{siegelset}
&\mathfrak S \ = \ \{\ g_\ell\cdot a \cdot g_r\ \mid \ g_\ell \in S_\ell\,,\ a \in A_\epsilon^+\,,\ g_r \in S_r\ \}\,,\, \ \text{with}
\\
&\ \ \ A_\epsilon^+\ =\ \{ \ a \in A\ \mid \ e^{\alpha}(a) > \epsilon \,\ \text{for all $\alpha \in \Phi^+(\fa,\fg)$}\,\}\,,\, \ \text{with}
\\
&\ \ \ \ \ \ \text{compact subsets}\,\ S_\ell \subset M \!\cdot \! N \,\ \text{and}\,\ S_r \subset G ,\ \text{and with}\,\ \epsilon > 0\,.
\end{aligned}
\end{equation}
This reduces to the usual notion of Siegel set in the  special case when $S_r = K$.

We do not mean to suggest that generalized Siegel sets $\mathfrak S$ have what Borel \cite[\S9.6]{Borel:1969b}
calls the ``Siegel property", i.e., the finiteness of the set of $\gamma\in \G$ such that $\gamma \!\cdot
\!\mathfrak S \cap \mathfrak S \neq \emptyset$ -- typically they do not. Siegel sets in the usual sense do have
this property, of course. For us, in the setting of smooth, but not necessarily $K$-finite, automorphic forms, the
notion of generalized Siegel set has significant technical advantages. In any case, our results below, on the rapid
decay of cuspidal automorphic forms on generalized Siegel sets, evidently imply rapid decay on Siegel sets in the
usual sense.

Siegel sets $\mathfrak S(S_\ell, \epsilon,K)$ in the usual sense have finite Haar measure, and there exists a
finite subset $C\subset G_{\Q}$ such that
\begin{equation}
\label{siegelset2}
\begin{aligned}
{\textstyle\bigcup}_{c\in C}\ \big(\Gamma\,c\,\mathfrak S(S_\ell,\epsilon,K)\big)\ = \ G\,,\,\ \ \text{provided $\epsilon$ and $S_\ell$}
\\
\text{are appropriately chosen}.
\end{aligned}
\end{equation}
Specifically, as $\,C\subset G_\Q\,$ one can take a complete set of representatives for the set of cusps, i.e., the
double coset space $\Gamma \backslash G_\Q / P_\Q$ \cite{Borel:1969b}. Thus all bounded functions $f \in
C(\Gamma\backslash G)$ are integrable over $\Gamma\backslash G$, in particular all functions $f \in
C(\Gamma\backslash G)$ that decay rapidly on any generalized Siegel set.

Because of the assumptions on $G$, there exists a finite dimensional representation $\tau : G \to GL(N,\R)$ whose
kernel is finite. For $g\in G$ we define
\begin{equation}
\label{norm-1}
\|g\| \ = \ \text{operator norm of $\tau(g)$\,}.
\end{equation}
The subgroup $A$ acts, via $\tau$, in a diagonalizable fashion, with positive real diagonal matrix entries. Thus,
for every $a\in A$, the norm  $\|a\|$ and the largest eigenvalue of $\tau(a)$ are mutually bounded. Recall that a
function $F\in C^\infty(G)$ is said to have {\em moderate growth\/} if
\begin{equation}
\label{moderategrowth}
|F(g)| \ \leq \ C\,\|g\|^N\ \ \ \text{for some $C>0$ and $N\in \N$}\,.
\end{equation}
It has {\em uniform moderate growth\/} if every $r(X)F$, with $X\in U(\fg_\C)$, has mode\-rate growth, with the
same index of growth $N$ as $F$ itself. Here $r$ denotes the action, by infinitesimal right translation, of the
universal enveloping algebra $U(\fg_\C)$ of the complexification $\fg_\C=\C\otimes_\R \C$ of $\fg$.

Let $(\pi,V)$ be an admissible representation of $G$, of finite length, on a reflexive Banach
space\begin{footnote}{We shall recall the definitions of admissibility and finite length in section
\ref{sec:proofmainthm}, and the relevance of reflexivity in the appendix.}\end{footnote}. The action $\pi$ of $G$
on the Banach space $V$ restricts to a representation on the space of $C^\infty$ vectors
\begin{equation}
\label{Vinfinity}
V^\infty\ = \ \{\,v\in V\, \mid \, g\mapsto \pi(g)v\,\ \text{is a $C^\infty$ function from $G$ to $V$}\,\}\,.
\end{equation}
It is dense in $V$ and carries a natural topology, the topology which $V^\infty$ inherits from its inclusion in the
space of $C^\infty$ $V$-valued functions on $G$ via (\ref{Vinfinity}). The action of $G$ on $V^\infty$ is
continuous with respect to this $C^\infty$ topology. We let $(\pi',V')$ denote the Banach dual of $(\pi,V)$, which
is again an admissible representation of finite length. Thus it makes sense to define the space of distribution
vectors for $V$ as
\begin{equation}
\label{V-infinity}
V^{-\infty}\ = \ \left((V')^\infty\right)',
\end{equation}
the continuous dual of the space of $C^\infty$ vectors $(V')^\infty$ for $(\pi',V')$. As follows formally from the
definition,
\begin{equation}
\label{Vinffinity-infinity}
V^\infty \ \subset \ V \ \subset \ V^{-\infty}\,,
\end{equation}
in analogy to the containments $C^{\infty}(X) \subset L^2(X) \subset C^{-\infty}(X)$ for a compact manifold $X$. By
duality, $G$ acts on $V^{-\infty}$. This action becomes continuous when one equips $V^{-\infty}$ with the strong
dual topology, which turns $V^{-\infty}$ into a complete, locally convex, Hausdorff topological vector
space\begin{footnote}{for a discussion of the topology on $(V')^{-\infty}$ and of the strong continuity of the
action of $G$ on $(V')^{-\infty}$ see the appendix.}\end{footnote}.

As a general fact \cites{Casselman:1989, wallach}, for any choice of $v\in V^\infty$ and $\tau\in (V')^{-\infty}$,
the complex valued function
\begin{equation}
\label{autodist-1}
g\ \mapsto \ F_{\tau,v}(g)\ =_{\text{def}} \ \langle \, \tau \,,\, \pi(g)v\,\rangle \ = \ \langle \, \pi'(g^{-1})\tau \,,\,v\,\rangle\
\end{equation}
is smooth, of uniform moderate growth~-- see (\ref{infty5a}) below~-- and transforms finitely under $Z(\fg_\C)$,
the center of the universal enveloping algebra $U(\fg_\C)$. We now fix an arithmetic subgroup $\G \subset G$. In
the preceding definition, if $\tau \in \left((V')^{-\infty}\right)^\G$, i.e., if $\tau\in (V')^{-\infty}$ is
$\G$-invariant, the function (\ref{autodist-1}) is left $\G$-invariant. Thus
\begin{equation}
\label{autodist0}
\tau \!\in\! \left((V')^{-\infty}\right)^\G\!,\,\ v\!\in\! V^\infty \ \ \Longrightarrow \ \ F_{\tau,v}\,\ \text{is a smooth $\G$-automorphic form}.
\end{equation}
This observation justifies calling $\tau \in \left((V')^{-\infty}\right)^\G$ a {\em $\G$-automorphic distribution
vector} for $V'$.

The most interesting examples of $\G$-invariant distribution vectors arise from cuspidal automorphic
representations, through a construction we now describe. The group $G$ acts unitarily on the Hilbert space
$L^2(\G\backslash G)$ by right translation. An {\it automorphic representation\/} consists of an irreducible
unitary representation $(\pi,V)$ of $G$, together with a $G$-invariant isometric embedding
\begin{equation}
\label{autorep}
j\, :\, V \ \hookrightarrow \ L^2(\G\backslash G)\,.
\end{equation}
If $v\in V^\infty$ is a $C^\infty$ vector, then $j(v)$ is a $\G$-invariant $C^\infty$ function on $G$. Thus one can
evaluate $j(v)$ at $e\in G$. The linear map
\begin{equation}
\label{autodist1} \tau_j\,:\, V^\infty\, \longrightarrow\, \C\,,\ \
\ \tau_j(v)\,=\,j(v)(e)
\end{equation}
is continuous with respect to the $C^\infty$ topology on the space of $V^\infty$, and it is also $\G$-invariant.
Thus, according to our terminology, $\tau_j$ is a $\G$-invariant distribution vector for the dual representation
$(\pi',V')$. We refer to $\tau_j$ as the automorphic distribution associated to the automorphic representation
(\ref{autorep}). It completely determines $j$, since $V^\infty$ is dense in $V$ and
$j(v)(g)=\pi(g)(j(v))(e)=\tau_j(\pi(g)v)$ for all $v\in V^\infty$, $\,g\in G$.  In the notation of
(\ref{autodist-1}), we have   $j(v)=F_{\tau_j,v}$, and hence

\begin{rem}
\label{ftaujv} \rm If $(\pi,V,j)$ is an automorphic representation as in (\ref{autorep}), the space $\{\,
F_{\tau_j,v} \,\mid \, v \in V^\infty \,\}\,$ coincides with the space of all $C^\infty$ vectors in the closed,
$G$-invariant, $G$-irreducible subspace $\,j(V) \subset L^2(\G \backslash G)$.
\end{rem}

The arithmetic subgroup $\G$ intersects any unipotent $\Q$-subgroup $U$ in an arithmetic subgroup $\G\cap U$, and
any arithmetic subgroup of $U$ is cocompact \cite{Borel:1969b}. Recall that a $\G$-automorphic form $F$ is said to
be {\it cuspidal\/} if
\begin{equation}
\label{cuspidal1}
\int_{(\G\cap U)\backslash U}\, F(u g)\,du\ = \ 0\ \ \ \text{for all $g\in G$}\,,
\end{equation}
and for any unipotent subgroup $U\subset G$ which arises as the unipotent radi\-cal of a proper $\Q$-parabolic
subgroup of $G$. The integral is well defined because $F$ is $\G$-invariant, and converges because $(\G\cap
U)\backslash U$ is compact. Analogously we call an automorphic distribution $\tau\in
\left((V')^{-\infty}\right)^\G$ cuspidal if
\begin{equation}
\label{cuspidal2} \int_{U/(\G\cap U)}\pi'(u)\tau\,du\ = \ 0\,,
\end{equation}
again for all subgroups $U\subset G$ which arise as the unipotent radical of a proper $\Q$-parabolic subgroup of
$G$. This latter integral makes sense because $V^{-\infty}$ is a complete, locally convex, Hausdorff topological
vector space, on which $G$ acts strongly continuously.

\begin{lem}\label{lem:cusp}
The following three conditions on the automorphic distribution $\tau_j$ corresponding to an automorphic
representation $j: V \hookrightarrow L^2(\G\backslash G)$ are equivalent:\vspace{-4pt}
\begin{itemize}
\item[\rm{a)}]\ $\tau_j$ is cuspidal;\vspace{-4pt}
\item[\rm{b)}]\  for all $\,v \in V^\infty$, the automorphic forms $\,j(v) = F_{\tau_j,v}$ are
    cuspidal;\vspace{-4pt}
\item[\rm{c)}]\ for some nonzero $\,v \in V^\infty$, $\,j(v) = F_{\tau_j,v}$ is cuspidal.
\end{itemize}
\end{lem}

\begin{proof}
Evidently b) implies c), and in view of (\ref{autodist-1}), a) implies b). If $F_{\tau_j,v_0}$ is cuspidal for some
non-zero $v_0\in V^\infty$, the continuous linear function from $V^\infty$ to $\C$, defined by
\begin{equation}
\label{cuspidal3} v\ \ \mapsto \ \ \langle \ \int_{U/(\G\cap U)}\pi'(u)\tau_j\,du\,\ ,\ v \,\ \rangle \ ,
\end{equation}
vanishes on all the $G$-translates of $v_0$. But those span a dense subspace of the irreducible $G$-module
$V^\infty$. The linear function (\ref{cuspidal3}) therefore vanishes identically, and (\ref{cuspidal2}) must hold.
In other words, c) implies a).
\end{proof}

As in the introduction, automorphic forms will be assumed to be smooth, but not necessarily $K$-finite. Also, again
as in the introduction, we do not expressly require automorphic forms to have uniform moderate growth as one
usually does, only moderate growth.

\begin{thm}
\label{thm:uniformlymg} Any cuspidal automorphic form $F$ has uniform moderate growth. Moreover, any such $F$ can
be expressed as a finite linear combination of cuspidal automorphic forms of type $F_{\tau_j,v}$\,, corresponding
to cuspidal automorphic distributions $\tau_j \in \left((V')^{-\infty}\right)^\G$ as in {\rm (\ref{autodist1})},
and $\,C^\infty$ vectors $v\in V^\infty$.
\end{thm}

Casselman has suggested that moderate growth of automorphic forms might imply uniform moderate growth even without
the hypothesis of cuspidality. We have no insight into this question. We shall prove the theorem in section
\ref{sec:uniform}. At this point, to prove theorem A, it suffices to establish the rapid decay on Siegel sets for
cuspidal automorphic forms of type $F_{\tau_j,v}$, as in the statement of the theorem. Let $\tau_j \in
\left((V')^{-\infty}\right)^\G$ be an automorphic distribution attached to an automorphic representation   The
topology on the space of $C^\infty$ vectors $V^\infty$ is Fr\'echet, and thus can be described by a family of
linear seminorms\begin{footnote}{The particular choice does not matter, but a concrete choice is described in
section \ref{sec:proofmainthm}.}\end{footnote} $\,\nu_k:V^\infty\to\R_{\geq 0}$\,, $k\in\N$\,. We may and shall
assume that the sequence is non-decreasing, in the sense that $\nu_{k+1}(v)\geq \nu_k(v)$ for all $v\in V^\infty$
and $k\in \N$.

\begin{thm}
\label{thm:decay1} Let $\mathfrak S = \mathfrak S(S_\ell,\epsilon,S_r)$ be a generalized Siegel set. If $\tau_j$ is
cuspidal, then for each $n\in\N$, there exists $k=k(\tau_j,n)\in\N$ and a positive constant $C=C(\tau_j,\mathfrak
S,k,n)$, such that for all $v\in V^\infty$,\vspace{-4pt}
\[
g \in \mathfrak S \ \ \ \ \Longrightarrow \ \ \ \ |F_{\tau_j,v}(g)| \ \leq \ C\,\nu_k(v)\,\|g\|^{-n} \,.
\]
\end{thm}

This implies theorem A, as was just mentioned, and also clarifies the dependence of the bounding constant $C$ on
the automorphic form being bounded. The proof of the theorem occupies section \ref{sec:proofmainthm}.

\section{The main theorems}\label{sec:maintheorems}

Theorems B\,--\,D of the introduction assert the rapid decay of cuspidal automorphic forms on subsets of $G$ that
are larger than Siegel sets. These three theorems constitute the main results of our paper. In the present section
we shall state slightly refined versions for automorphic forms of type $F_{\tau_j,v}$, and deduce them from theorem
\ref{thm:decay1}. Just as theorem \ref{thm:decay1} implies theorem A via theorem \ref{thm:uniformlymg}, so the
theorems of this section imply theorems B\,--\,D of the introduction. We continue with the notation of the previous
section. In particular, $\tau_j \in \left((V')^{-\infty}\right)^\G$ is the automorphic distribution attached to an
automorphic representation $j: V \hookrightarrow L^2(\G\backslash G)$, and $\,\nu_k$\,, $k\in\N$\,, a
non-decreasing family of seminorms describing the topology of $V^\infty$.

\begin{thm}
\label{thm:decay2} Let $S_r$ be a compact subset of $G$. If $\tau_j$ is cuspidal, then for each $n\in\N$, there
exists an integer $k=k(\tau_j,n)\in\N$ and a positive constant $C=C(\tau_j,S_r,k,n)$, such that for all $v\in
V^\infty$,\vspace{-4pt}
\[
m\in M\,,\,\ a \in A\,,\,\ g\in S_r\ \ \ \ \Longrightarrow\ \ \ \ |F_{\tau_j,v}(mag)| \ \leq \ C\,\nu_k(v)\,\|a\|^{-n} \,.
\]
\end{thm}

\begin{proof}
The action of $\G\cap M$ on the anisotropic group $M$ has a compact fundamental domain $F\subset M$, which allows
us to replace $m\in M$, modulo $\G$, by some $m\in F$. Since $M$ commutes with $A$, we can move the factor $m\in F$
to the right of $A$ and incorporate it into the compact set $S_r$. In other words, we may as well suppose $m=e$.
When $a$ lies in the positive Weyl chamber, the products $ag$, with $g\in S_r$, all lie in the generalized Siegel
set $\mathfrak S(\{e\},1,S_r)$. A finite number translates of the positive Weyl chamber, by elements $n_\ell$ of
the normalizer of $A$ in $G_\Q$, cover all of $A$. Each of the factors $n_\ell^{-1}$ to the right of $a$ can be
incorporated into $S_r$ by enlarging this set. As a factor to the left of $a$, $n_\ell$ disappears when we
substitute $\pi'(n_\ell^{-1})\tau_j$ for $\tau_j$, or equivalently, replace the automorphic representation
(\ref{autorep}) by its own $n_\ell$-conjugate. That conjugate is automorphic with respect to $n_\ell^{-1}\G
n_\ell$, which is another arithmetic subgroup of $G$. Note also that $\|a\|$ and $\|a\,g\|$, with $g$ ranging over
the compact set $S_r$, are mutually bounded. Thus, when we apply theorem \ref{thm:decay1} to the finitely many
translates $\pi'(n_\ell^{-1})\tau_j$, we obtain the estimate we want.
\end{proof}

\begin{rem}
\label{rem:decay2} \rm Theorem \ref{thm:decay2} is not only implied by theorem \ref{thm:decay1}, but also implies
it. Any $\,a\in A_\epsilon^+\,$ acts on $\,\fn\,$ with eigenvalues that are bounded from below. It follows that
$\,a^{-1}\,S_\ell\,a\,$ is contained in a fixed compact set, uniformly for $\,a\in A_\epsilon^+\,$. This set can be
absorbed into $\,S_r\,$.
\end{rem}

Our next statement involves a parabolic subgroup $P_1 \subset G$, which is defined over $\Q$. Since the minimal
$\Q$-parabolic subgroup $P$ is determined only up to $G_{\Q}$-conjugacy, we may as well suppose that $P_1 \supset
P$. Then $P_1$ has Langlands decomposition
\begin{equation}
\begin{aligned}
\label{pone}
&P_1 \ = \ M_1\!\cdot\! A_1 \! \cdot \! N_1\,,\, \ \text{with}\, \ M_1 \supset M\,,\, \ A_1 \subset A\,,\, \ N_1 \subset N\,,\, \ \text{and}
\\
&\ \ N_1\, \ \text{unipotent, connected, defined over $\Q$, normalized by $M_1\!\cdot\! A_1$,}
\\
&\ \ \ \ A_1\, \ \text{abelian, connected, $\Q$-split, central in $M_1\! \cdot\! A_1$\,,}
\\
&\ \ \ \ \ \ M_1\ \ \text{reductive, defined over\,\ $\Q$\,;\,\ moreover, $\,M_1\,$ contains}
\\
&\ \ \ \ \ \ \ \ M_1 \cap P = M\!\cdot\!(M_1 \cap A)\!\cdot\!(M_1 \cap N)\,\ \text{as a minimal $\Q$-parabolic.}
\end{aligned}
\end{equation}
The Lie algebra $\,\fm_1\cap \fa\,$ of $\,M_1\cap A\,$ is spanned by the co-roots $\,H_\a$ corres\-ponding to all
roots $\,\a\in \Phi(\fa,\fm_1)$, so $\,\fm_1\cap \fa \subset \fm + [\fm_1,\fm_1]$. In particular the identity
component of the center of $\,M_1\,$ is contained in the anisotropic group $\,M\,$, and
\begin{equation}
\label{pone1}
\text{$\,M_1\,$ inherits the property (\ref{cptctr}) from $\,G\,$.}
\end{equation}
On the other hand $\fn_1$, the Lie algebra of $N_1\,$, is the direct sum of all the root spaces $\fg^{\alpha}$
which do not lie in $\fm_1$, and which correspond to positive roots $\alpha$. Equivalently, these are the roots
$\alpha\in \Phi^+(\fa,\fg)$ which restrict nontrivially to $\fa_1$. The linear function
\begin{equation}
\label{rhoone}
\rho_1\ : \ \fa \ \longrightarrow \ \R\ ,\ \ \  \rho_1 \ = \
{\sum}_{\textstyle{\begin{smallmatrix}\alpha\in {\Phi^+(\fa,\fg)} \\  { \alpha \equiv \!\!\!\!/ \, 0 \ \text{on}\ \fa_1} \end{smallmatrix}}}\
\textstyle{\frac {\dim(\fg^\alpha)}2}\,\alpha  \ ,
\end{equation}
lifts to a character $\,e^{\rho_1} : A \to \R_{>0}$\,. Its square $e^{2\rho_1}$ is the character by which $A$ acts
on $\wedge^{\text{top}}\fn_1$, the top exterior power of $\fn_1$. Since $P_1$ is normalizer of $N_1$, and since the
adjoint action of $P_1$ on $\fn_1$ is defined over $\Q$, $e^{2\rho_1}$ extends from $A$ to a character
\begin{equation}
\label{rhotwo}
e^{2\rho_1}\ : \ P_1 \ \longrightarrow \ \R^*\,,
\end{equation}
also defined over $\Q$. The absolute value of this character,
\begin{equation}
\label{deltaPbeta}
\delta_{P_1} \ : \ P_1 \ \longrightarrow \ \R_{>0}\,, \ \ \ \ \delta_{P_1} \, = \,  |\,e^{2\rho_1} | \,,
\end{equation}
is the modular function for the quotient $G/P_1$\,.

\begin{thm}\label{thm:decay3}
If $\tau_j$ is cuspidal and $S_r \subset G$ a compact subset, then
\newline
\noindent {\rm a)} for each $n\in \N$ there exist $\,k=k(\tau_j,n)\in \N$ and $C=C(\tau_j,S_r,k,n)>0$ such that,
for all $v\in V^\infty$,
\[ p \in P_1\,,\,\ g \in S_r \
\ \ \Longrightarrow\ \ \ \ |F_{\tau_j,v}(p\,g)| \ \leq \ C\,\nu_k(v) \( \delta_{P_1}(p)\)^{-n} \,;
\]
{\rm b)} for each $n\in \N$ there exist $k=k(\tau_j,n)\in \N$ and $D=D(\tau_j,v, S_r,n)>0$ such that, for all $v\in
V^\infty$,
\[
m \!\in\! M,\, n_1 \!\in\! N_1,\, a\! \in\! M_1 \!\cap\! A,\, g \!\in\! S_r
\ \ \Longrightarrow \ \ |F_{\tau_j,v}(m\, n_1 \,a\,g)| \, \leq \, D \nu_k(v)\|a\|^{-n} \,.
\]
\end{thm}

In view of theorem \ref{thm:uniformlymg}, part a) implies theorem D and part b) implies theorem C.

\begin{rem}
\label{rem:decay3} \rm The theorem applies in particular to $\,P_1 = P\,$ and to $P_1 = G$. In the first case, {\rm
a)} provides a bound on $\,|F_{\tau_j,v}(g)|\,$ for $\,g\in S_\ell\,A_\epsilon^+\,S_r\,$, in terms of
$\,e^{-2\rho}(a)\,$, evaluated on the $\,A_\epsilon^+$-component $\,a\,$ of $\,g\,$; here $\rho$ is the half sum of
all positive roots, each counted with its multiplicity. Since $2\rho$ is dominant and regular, $\,e^{-2\rho}(a)\,$
with $a\in A_\epsilon^+$ can be bounded from above by a multiple of a negative power of $\,\|a\|\,$. Since $S_\ell$
and $S_r$ are compact, the norm $\,\|a\|\,$ in turn can be bounded from above and below in terms of $\,\|g\|\,$.
Thus a), with $P_1=P$, implies theorem  \ref{thm:decay1}. When $P_1=G\,$,\,\ {\rm b)} implies theorem
\ref{thm:decay2} directly, and theorem \ref{thm:decay2} is equivalent to theorem \ref{thm:decay1} by remark
\ref{rem:decay2}. Thus theorem \ref{thm:decay1} is not only used to prove the two parts of theorem
\ref{thm:decay3}, but is also implied by each of them.
\end{rem}

We shall reduce the two estimates in the theorem to theorem \ref{thm:decay1}, using some results in reduction
theory due to Borel and Harish Chandra \cite[\S\S14-16]{Borel:1969b}. These results have typically been applied to
the fundamental representations of the complexified Lie algebra $\fg_\C\,$;\,\ we apply them to suitably chosen
larger representations to obtain the quite general estimates of theorem \ref{rem:decay3}.

We begin with two preliminary lemmas, which may be of independent interest. Let $\mu\in \fa^*$ be a $\Q$-weight,
i.e., the differential of the restriction to $A$ of a $\Q$-character
\begin{equation}
\label{weight1}
e^\mu \, : \, M \! \cdot \! A \ \longrightarrow \ \R^* \,,\ \ \text{with}\ \ e^\mu \, \equiv \, 1 \ \ \text{on}\ \ M^0 \,.
\end{equation}
The notation $e^\mu$ might be taken to suggest that $\mu$ completely determines the character (\ref{weight1}). This
is almost true: any $\Q$-character of the anisotropic group $M$ is trivial on the identity component
\cite[10.5]{Borel:1969b}, so $\mu$ determines $e^\mu$ on the identity component of $M\! \cdot \! A$, which is the
most one can expect. For the same reason the second condition in (\ref{weight1}), i.e., $e^\mu \equiv 1$ on $M^0$,
is auto\-matically satisfied. Since $M\! \cdot \! A$ normalizes $N$, we can extend $e^\mu$ to a character
\begin{equation}
\label{weight1b}
e^\mu \, : \, P \ \longrightarrow \ \R^* \,,
\end{equation}
also defined over $\Q$, which is trivial on the connected unipotent group $N$.

In the following we consider a particular character $e^\mu$ as in (\rangeref{weight1}{weight1b}), and we impose the
additional condition that $\mu$ is dominant, in the sense that
\begin{equation}
\label{weight2}
\textstyle (\alpha,\mu)\geq 0\ \ \text{for all}\ \ \alpha\in\Phi^+(\fa,\fg)\,.
\end{equation}
Then there exists a finite dimensional representation of $G$ on a real vector space $V_\mu$, defined over $\Q$, and
irreducible over $\Q$, of highest weight $\mu$ \cite[\S14.1]{Borel:1969b}. In other words, $V_\mu$ contains a line
$L_\mu$, defined over $\Q$, such that
\begin{equation}
\begin{aligned}
\label{weight3}
&p\cdot v_\mu \ = \ e^\mu(p)\,v_\mu \ \ \text{for\ \ $v_\mu \in L_\mu$ and all $p\in P$\,,\,\ and}
\\
&\text{no $\,\mu + \alpha\,$,\,\ with $\,\alpha\in\Phi^+(\fa,\fg)\,$,\,\ is a weight for $\,V_\mu\,$.}
\end{aligned}
\end{equation}
We equip $V_\mu$ with a $K$-invariant metric, normalize $v_\mu\in L_\mu$ to make it a unit vector, and define
\begin{equation}
\label{phimu}
\phi_\mu \, : \, G\ \longrightarrow\ \R_{>0}\ ,\ \ \ \phi_\mu(g)\,=\, \|\,g\cdot v_\mu\|\,.
\end{equation}
Borel \cite{Borel:1969b} calls this a function of type $(P,e^\mu)$, and uses the notation $\Phi$. Borel lets
$\Gamma$ act on $G$ on the right, which accounts for some differences between his formulas and ours.

The following is a slight variant of Borel \cite[theorem 16.9]{Borel:1969b}. In the case of $G=SL(n,\R)$, equipped
with the standard $\Q$-structure, and when $\mu$ is a fundamental highest weight, it is a fairly standard result in
reduction theory.

\begin{lem}
\label{lem:borel16.9} Given a finite subset $C_0\subset G_\Q$, one can choose a finite subset $C \subset G_\Q$
containing $C_0$ and a generalized Siegel set $\mathfrak S = \mathfrak S(S_r,\epsilon,S_\ell)$, having the
following property: for every $g\in G$ the function
\[
\,\Gamma \! \times \! C  \to  \R_{>0}\,,\ \ \ \ \ (\gamma, c) \ \mapsto \ \phi_\mu (g^{-1}\gamma\,c)\,,
\]
assumes its minimum at some $\, (\gamma,c) \in \Gamma \! \times \! C\, $ such that $\, g \in \gamma\,c\,\mathfrak
S$.
\end{lem}

In \cite{Borel:1969b}, the role of $C_0$~-- there denoted by $C$~-- is played by any finite subset $C_0\subset
G_\Q$ containing a complete set of representatives of the set of cusps $\G\backslash G_\Q/P_\Q$. In our application
below what matters is only that the set $C$ in the lemma contains the identity.

\begin{proof} Borel states this estimate in terms of functions $\phi : G \to \R_{>0}$ which are comparable, but not
necessarily equal, to the function $\phi_\mu$ of (\ref{phimu}). On the other hand, he does not require the
triviality of the character $e^\mu$ on all of $Z_G^0$, but only on $Z_G^0\cap [M,M]$. Borel therefore needs to
impose two additional hypotheses: the function $\phi = \phi_\mu$ must satisfy
\begin{equation}
\label{borel16.9hypo}
\begin{alignedat}{2}
&\text{i)}\,\  &&\text{$\phi(g\,\gamma) = \phi(g)\,$ for every $\,\gamma \in \Gamma \cap P\,$, and}
\\
&\text{ii)}\,\ &&\text{$r(X)\phi = 0$ for every $X\in\fg^\alpha$,\,\ provided $\alpha\in \Phi(\fa,\fg)$ lies in}
\\
&\  &&\text{the $\Z$-linear span of the simple roots $\beta$ such that $(\mu,\beta)=0$\,};
\end{alignedat}
\end{equation}
here, as before, $r(X)$ denotes infinitesimal right translation by $X$. Note that ii) is the infinitesimal version
of right invariance of the function $\phi_\mu$ under the action of the connected Lie group which Borel denotes by
$L_{\theta'}$. Our lemma will follow directly from Borel \cite[theorem 16.9]{Borel:1969b} once we show that the
function $\phi_\mu$ defined in (\ref{phimu}) automatically satisfies these two conditions i), ii).

The line $L_\mu$ in (\ref{weight3}) is defined over $\Q$, invariant under $P$, and therefore invariant under the
arithmetic subgroup $\Gamma \cap P$. The resulting homomorphism $\Gamma \cap P \to GL(L_\mu)$ must then factor
through $GL(1,\Z) \cong \pm 1$, which implies i). To establish ii), we complexify $V_\mu$ and $\fg$, and we extend
$\fa$ to a Cartan subalgebra $\fb \subset \fg$. Every $\alpha \in \Phi(\fa,\fg)$ arises as the restriction to $\fa$
of some $\eta \in \Phi(\fb_\C,\fg_\C)$, the root system of the complex reductive Lie algebra $\fg_\C = \C\otimes_\R
\fg$ with respect to its Cartan subalgebra $\fb_\C = \C \otimes_\R\fb$. Because of (\ref{weight3}) every root space
$\fg^\alpha$, with $\alpha \in \Phi^+(\fa,\fg)$, must annihilate $v_\mu$. Thus, if ii) were to fail, there would
exist a root $\eta \in \Phi(\fb_\C,\fg_\C)$, which restricts to a positive root $\alpha = \eta|_{\fa}\in
\Phi^+(\fa,\fg)\cap \mu^\perp$, and a root vector $E_{-\eta}$ in the one dimensional $(-\eta)$-root space
relatively to the Cartan subalgebra $\fb_\C$, such that $E_{-\eta}v_\mu \neq 0$. If we choose a generator $E_\eta$
of the $\eta$-root space and scale it appropriately, the triple $E_\eta$, $E_{-\eta}$, $H_\eta =
[E_\eta,E_{-\eta}]$ satisfies the commutation relations $[H_\eta, E_{\pm \eta}] = \pm 2 E_{\pm \eta}$. Also
$H_\eta\in \fb$ corresponds to $2(\eta,\eta)^{-1}\eta \in \fb^*$ via the isomorphism $\fb \cong \fb^*$ determined
by any $Ad$-invariant symmetric bilinear form on $\fg_\C$. Because of (\ref{weight1}), $\mu$ restricts trivially to
$\fb_\C \cap \fm_\C$. Also, $\eta|_\fa = \alpha$ by definition, and $\alpha$ was chosen so that $\alpha \perp \mu$.
Thus $(\eta,\mu)=0$, which now implies $H_\eta \,v_\mu = 2\frac{(\eta,\mu)}{(\eta,\eta)} v_\mu =0$. On the other
hand $E_\eta \in \C\otimes_\R\fg^\alpha$, so $E_\eta \,v_\mu = 0$ by (\ref{weight3}). If both $E_{\eta}$ and
$H_\eta$ annihilate a vector in a finite dimensional representation of $SL(2,\R)$, then so does $E_{-\eta}$ -- see,
for example, \cite[theorems III.8, III.12]{Jacobson}. This contradicts the earlier assumption that
$E_{-\eta}\,v_\mu \neq 0$, and establishes ii) by contradiction. The lemma thus follows from \cite[theorem
16.9]{Borel:1969b}.
\end{proof}

\begin{lem}
\label{lem:phimulem} Let $\tau_j$ be cuspidal, $S_r \subset G$ a compact subset, and $c \in G_\Q$. For each $n\in
\N$ there exist $k=k(\tau_j,n)\in\N$ and $D=D(\tau_j,S_r,c,n,k)>0$ such that $F_{\tau_j,v}(c^{-1}g\,g_r) \leq
D\,\nu_k(v)\, \phi_\mu(g^{-1})^n$, for all $g\in G$, $g_r \in S_r$, $v\in V^\infty$.
\end{lem}

\begin{proof}
Let us argue, first of all, that it suffices to establish this bound for $c=e$ and $S_r = \{e\}$. If we replace $g$
by $g\,g_r^{-1}$, we must bound $|F_{\tau_j,v}(c^{-1}g)|$ in terms of $\phi_\mu(g_r\,g^{-1})^n =
\|g_r\,g^{-1}\!\cdot \! v_\mu\|^n$. As $g_r$ ranges over the compact set $S_r$, its operator norm, acting on
$V_\mu$, is uniformly bounded from above and below. Thus, after modifying $D$ appropriately, we may set $g_r = e$.
The estimate to be proved does not involve $\G$ explicitly. Thus, once it is proved for any $\G$ but only with $c =
e$, it will also apply to the function $g \mapsto F_{\tau_j,v}(c^{-1}g)$, which is automorphic with respect to the
arithmetic group $\, c\,\G c^{-1}\,$. Thus we may also suppose that $c=e$.

For the reasons just mentioned, theorem \ref{thm:decay1} applies not only to $F_{\tau_j,v}$, but also the functions
$g \mapsto F_{\tau_j,v}(c^{-1}g)$ with $c$ ranging over any given finite subset $C\subset G_\Q$. Thus, when we
choose $C$ and $\mathfrak S$ as in  lemma \ref{lem:borel16.9}, then for every $n\in \N$ there exist
$k=k(\tau_j,n)\in\N$ and $D_1
>0$, such that
\begin{equation}
\label{normbound}
g \in \mathfrak S\,, \ c\in C\,, \ v\in V^\infty \ \ \Longrightarrow \ \ |F_{\tau_j,v}(c\, g)| \, \leq \, D_1\,\nu_k(v)\,\|g\|^{-n}\,.
\end{equation}
According to our original hypotheses, $G$ is a finite cover of a subgroup $G_{\text{lin}}$ of $GL(N,\R)$, and the
norm $\|g\|$ in (\ref{normbound}) is the operator norm of $g$ acting on $\R^N$. The representation of $G$ on
$V_\mu$ drops to an algebraic homomorphism from $G_{\text{lin}}$ to $GL(V_\mu)$, and this implies that the operator
norm of $g \in G$ acting on $V_\mu$ can be bounded in terms of some positive integral power\begin{footnote}{This
depends on our hypothesis (\ref{cptctr}): the operator norm of $a\in GL(1,\C)\cong \C^*$ acting by the tautological
representation is the absolute value $|a|$, and no multiple of $|a|^m$, $m\in\N$,  bounds $|a|^{-1}$, the operator
norm under the dual representation, simultaneously for all $a\in\C^*$.}\end{footnote} of $\|g\|$. Thus $\phi_\mu(g)
= \|g\! \cdot \! v_\mu\|$ can be bounded from above by a multiple of $\|g\|^m$ for some $m\in \N$. On the other
hand, $1 = \|v_\mu\| = \|\,g^{-1} \cdot g \cdot v_\mu\,\|\,$, so $\phi_\mu(g)$ can be bounded from below by the
reciprocal of the operator norm of $g^{-1}$, acting on $V_\mu$, and hence in terms of a negative power of
$\|g^{-1}\|$. To summarize, there exist $D_2>0$ and $m\in\N$ so that for all $g\in G$,
\begin{equation}
\label{normbound1}
D_2^{-1}\,\|g^{-1}\|^{-m}\ \leq \ \phi_\mu(g) \ \leq \ D_2\,\|g\|^m\,.
\end{equation}
Now let $g$ be given. According to lemma \ref{lem:borel16.9}, we can choose $c_1\in C$, $\gamma\in \G$, and $h \in
G$ such that
\begin{equation}
\label{normbound3}
g \, = \, \gamma\,c_1\,h\,,\ \ h\in \mathfrak S\,,\ \ \text{and}\ \ \phi_\mu(h^{-1})\, \leq \,\phi_\mu(g^{-1})\,.
\end{equation}
Thus, view of (\ref{normbound}) and (\ref{normbound1}),
\begin{equation}
\begin{aligned}
\label{normbound4}
&|F_{\tau_j,v}(g)| \, = \, |F_{\tau_j,v}(\gamma \, c_1\, h)| \,=\, |F_{\tau_j,v}(c_1\, h)| \, \leq \, D_1\,\nu_k(v)\,\|h\|^{-n}
\\
&\qquad\ \ \ \leq \, D_1\,D_2^{\,n/m}\,\nu_k(v)  \,\phi_\mu(h^{-1})^{n/m}\,\leq \, D_1\,D_2^{\,n/m} \,\nu_k(v) \,\phi_\mu(g^{-1})^{n/m}\,.
\end{aligned}
\end{equation}
The lemma follows: after adjusting $k$ and $D_1$, one can replace $\,n\,$ by $\,n\,m\,$.
\end{proof}

\noindent {\it Proof\/} of theorem \ref{thm:decay3}.\ \ It is customary to let $\rho$ denote the half sum of the
positive roots. In our situation, since the root spaces may have dimension greater than one, it is more convenient
to set
\begin{equation}
\label{rhodef}
\rho \ \ = \ \ {\sum}_{\alpha\in\Phi^+(\fa,\fg)} \,\textstyle \frac{\dim \fg^\alpha}{2} \,\alpha \ ,
\end{equation}
which is consistent with the definition (\ref{rhoone}) of $\rho_1$. We also define
\begin{equation}
\label{rhomonedef}
\rho_{M_1} \ \ = \ \ {\sum}_{\alpha\in\Phi^+(\fa,\fm_1\oplus\fa_1)} \,\textstyle \frac{\dim \fg^\alpha}{2} \,\alpha\,;
\end{equation}
here $\Phi^+(\fa,\fm_1\oplus\fa_1) = \Phi(\fa,\fm_1\oplus\fa_1)\cap\Phi^+(\fa,\fg)$ denotes the set of positive
roots for the root system of $\fa$ acting on $\fm_1\oplus\fa_1$. These are precisely the roots in $\Phi^+(\fa,\fg)$
that vanish identically on $\fa_1$. Hence
\begin{equation}
\label{sumofrho}
\rho \ \ = \ \ \rho_1 \ +\ \rho_{M_1}\,.
\end{equation}
We had remarked already that $2\rho_1$ lifts to an algebraic character of $P_1$, defined over $\Q\,$; this is the
character by which $P_1$ acts on $\wedge^{\text{top}}\fn_1$. As was remarked below (\ref{weight1}), any
$\Q$-character is trivial on the identity component of the anisotropic group $M$. Recall also that $P_1$ contains
$P$. Thus $2\rho_1$ can play the role of $\mu$ in (\rangeref{weight1}{weight1b}).

Let $\Delta$ denote the set of simple roots in $\Phi^+(\fa,\fg)$. With the convention (\ref{rhodef}), the familiar
identity $2(\rho,\beta)=(\beta,\beta)$ for $\beta\in\Delta\,$ gets replaced by
\begin{equation}
\label{rhoprop}
2\, \frac{(\rho,\beta)}{(\beta,\beta)} \ \ = \ \ \dim \fg^\beta \, + \, \dim \fg^{2 \beta} \ \ \ (\,\beta \in \Delta)\,.
\end{equation}
If $2\beta$ is not a root, $\dim \fg^{2\beta}$ should be interpreted as zero, of course. We can express
$\,\Delta\,$ as the disjoint union
\begin{equation}
\label{Delta}
\begin{aligned}
\Delta \ = \ \Delta_1 \cup \Delta_{M_1}\,,\ \ \text{with} \ \ \Delta_1 \,=\, \{\,\beta\in \Delta\,\mid\, \beta|_{\fa_1}\neq 0\,\}\,,
\\
\Delta_{M_1} \,=\, \{\,\beta\in \Delta\,\mid\, \beta|_{\fa_1} = 0\,\}\,.
\end{aligned}
\end{equation}
Equivalently $\Delta_{M_1}$ is the set of simple roots in $\Phi^+(\fa,\fm_1 \oplus \fa_1)$. Since $[M_1,M_1]$ must
act trivially on the one dimensional space $\wedge^{\text{top}}\fn_1$,
\begin{equation}
\label{rhoprop1}
\beta \in \Delta_{M_1} \ \ \ \ \Longrightarrow\ \ \ \ (\beta,\rho_1)\,=\,0\,.
\end{equation}
On the other hand, every $\beta\in\Delta_1$ has a non-positive inner product with any other simple root, hence with
any $\eta\in\Delta_{M_1}$. Every $\alpha\in\Phi^+(\fa,\fm_1\oplus\fa_1)$ is an integral linear combination, with
non-negative integral coefficients, of the $\eta\in\Delta_{M_1}$, so $\beta\in\Delta_1$ implies
$(\beta,\rho_{M_1})\leq 0$. Thus, in view of (\ref{sumofrho}) and (\ref{rhoprop1}), $2\rho_1$ is dominant.

Since we have shown that $2\rho_1$ satisfies all the relevant hypotheses (\rangeref{weight1}{weight2}), we can
apply lemma \ref{lem:phimulem} with $\mu = 2\rho_1$. We claim: the identity $\,p\!\cdot\,\!v_\mu = e^\mu(p)\,v_\mu$
in (\ref{weight3}), with $\mu=2\rho_1$, holds not only for $p\in P$, but even for $p\in P_1^0$. Indeed, part ii) of
(\ref{borel16.9hypo}) ensures that $[M_1,M_1]^0$ acts trivially on $v_{2\rho_1}$. This implies the claim, because
$P^0$ and $[M_1,M_1]^0$ generate $P_1^0$. The finite group $P_1/P_1^0$ must then act by $\pm 1$ on the invariant
line $L_{2\rho_1}$, so the identity $\,p\cdot v_{2\rho_1} = e^\mu(p)\,v_{2\rho_1}$ holds at least up to sign for
all $p\in P_1$. Part a) of the theorem now follows from (\ref{deltaPbeta}), (\ref{phimu}), and lemma
\ref{lem:phimulem}.

The simple roots constitute a basis of $\fa^*$. For $\beta\in\Delta\,$, we can therefore define $\mu_\beta\in\fa^*$
by the identities
\begin{equation}
\label{mubeta}
\textstyle 2\frac{(\,\eta \,,\, \mu_\beta\,)}{(\,\eta \,,\, \eta\,)} \ = \ \delta_{\beta,\eta} \! \cdot \! \left(\, \dim \fg^\beta + \dim \fg^{2 \beta} \,\right) \ \ \ \text{for all $\,\eta\in\Delta$}\,.
\end{equation}
In the following, $\ell$ denotes a strictly positive integer that will be specified later. The quantity
\begin{equation}
\label{mu1}
\mu \ \ = \ \ \ell \cdot{\sum}_{\beta\in\Delta_{M_1}}\, \mu_\beta
\end{equation}
satisfies
\begin{equation}
\label{mu2}
\begin{aligned}
&\beta\in\Delta_1 \ \ \ \Longrightarrow \ \ \ (\,\mu\,,\,\beta\,) \ = \ 0\,,\ \ \ \text{and}
\\
&\beta\in\Delta_{M_1} \ \ \Longrightarrow \ \ \textstyle 2\frac{(\,\mu\,,\,\beta\,)}{(\,\beta\,,\,\beta\,)} \, =
\, \ell\,(\dim \fg^\beta + \dim \fg^{2 \beta}) \,
= \, 2\ell\frac{(\,\rho_{M_1}\,,\,\beta\,)}{(\,\beta\,,\,\beta\,)}\,;
\end{aligned}
\end{equation}
cf. (\rangeref{sumofrho}{rhoprop}) and (\rangeref{rhoprop1}{mubeta}). In particular,
\begin{equation}
\label{mu3}
\mu \ \ \text{and}\ \ \ell\,\rho_{M_1}\ \ \text{have the same restriction to}\ \ \fa \cap \fm_1 \,,
\end{equation}
since $\,\Delta_{M_1}\,$ spans the dual space of $\,\fa\cap \fm_1\,$.

The quantity $\,\mu\in\fa^*\,$ is defined over $\,\Q\,$. Thus, if $\ell$ is chosen appropriately, $\mu$ lifts to an
algebraic character $e^\mu$ on the Zariski closure of $A$, a character that is defined over $\,\Q\,$. We now
replace this choice of $\ell$ by a multiple of itself so that $e^\mu$ becomes identically equal to one on the
intersection of $M$ with the Zariski closure of $A$. It can then extended to $M\!\cdot \!A \! \cdot \! N = P$,
trivially across $M\! \cdot \! N$. To summarize, $\ell$ can be chosen so that $\mu$ lifts to an algebraic character
\begin{equation}
\label{mu4}
e^\mu \, : \, P \,\rightarrow \, \R^*\,,\ \ \text{defined over $\,\Q\,$,\,\ with}\ \ e^\mu \,\equiv\, 1 \ \ \text{on}\ \ M\!\cdot \! N\,.
\end{equation}
With this choice of $\ell$, the character $e^\mu$ satisfies the hypotheses of lemma \ref{lem:phimulem}, in
particular (\rangeref{weight1}{weight2}). Hence the lemma, in conjunction with (\ref{weight3}), (\ref{phimu}),
(\ref{mu3}), and (\ref{mu4}) implies: for $c\in G_\Q\,,\ n_1 \in N_1\,,\ a\in M_1\cap A\,,\ g\in S_r$\,,
\begin{equation}
\label{mu5}
|F_{\tau_j,v}(c^{-1}\, n_1 \,a\,g)|  \ \leq \ D\,\nu_k(v)\,(\,e^{-\ell\rho_{M_1}}(a)\,)^{n}\,.
\end{equation}
For $H$ in the positive Weyl chamber $(\fm_1 \cap \fa)^+$, $\langle \,\rho_{M_1},H\,\rangle $ is bounded from below
by a positive multiple of $\|H\|$. On the group level, it follows that $e^{-\ell\rho_{M_1}}(a)$, with $a$ in the
positive Weyl chamber $(M_1 \cap A)^+\subset(M_1 \cap A)$, is bounded from above by some negative power of $\|a\|$.
Thus (\ref{mu5}) implies the estimate in b), at least for $m=e$ and $a \in (M_1 \cap A)^+$. Since $M_1$ normalizes
$N_1$, we can let $c$ run over a complete set of representatives for the Weyl group of $M_1\cap A$ in $M_1$, to
extend the validity of the estimate to all $a \in (M_1 \cap A)$; the argument is the same as the one in the proof
of theorem \ref{thm:decay2}. It remains to justify the presence of the factor $m\in M$ in b). Since $M$ is
anisotropic, $\G \cap M$ is cocompact in $M$, so it suffices to consider $m$ in a compact subset of $M$. Since $M$
normalizes $N_1$ and commutes with $A$, we can move the factor $m$ across $n$ and $a$ and incorporate it into a
larger, but still compact subset $S_r\subset G$. \qed

\section{Unipotent averaging}\label{sec:unipotent}

We continue with the hypotheses and notation of the previous section. Some derivations of functional equations for
$L$-functions involve averaging automorphic forms against a character on the unipotent radical $N_1$ of a parabolic
subgroup $P_1 \subset G$, followed by integration over a subgroup of the Langlands-Levi component $M_1$. Ginzburg
\cite[p.\,102]{Ginzburg:Crelle1995} discusses this situation from a formal point of view. We begin with a
description of the general setting.

Let $\,P_1 = M_1\!\cdot \! A_1\!\cdot \! N_1\,$ be a $\Q$-parabolic subgroup of $G$, as in (\ref{pone}). The
conjugation action of $M_1$ on $N_1$ induces an action of $M_1$ on the group of unitary abelian characters
$\,\widehat{N_1}\,$. In the following, we fix a $\,\psi \in \widehat{N_1}\,$ such that
\begin{equation}
\label{psi1}
\psi \equiv 1 \,\  \text{on}\ \,\Gamma \cap N_1\,.
\end{equation}
Then one can define the averaging operator
\begin{equation}
\label{psi2}
I_\psi  :   C^\infty(\Gamma\backslash G)\ \to \ C^\infty(G)\,, \, \
(I_\psi F)(g)\, =  \int_{(\Gamma\cap N_1)\backslash N_1}\!\! \psi(u^{-1}) \, F(u\,g)\,du \ ;
\end{equation}
here Haar measure on $N_1$ is to be normalized so that $\,(\Gamma\cap N_1)\backslash N_1\,$ has total measure one.
Since $\,P_1\,$ is a $\Q$-parabolic subgroup, its unipotent radical $\,N_1\,$ is defined over $\,\Q\,$. Via the
exponential map the Lie algebra $\,\fn_1\,$ inherits a rational structure $\,\fn_{1,\Q}\,$. In view of
(\ref{psi1}), the values of the differential $\,\psi_*\,$ on the set $\,\log(\Gamma\cap N_1)\,$ lie in $\,2 \pi
i\,\Z\,$. In particular $\,\frac{1}{2\pi i}\,\psi_*\,\in \fn_{1,\Q}^*\,$, so
\begin{equation}
\label{psi3}
\begin{aligned}
&M_{1,\psi} \ =_{\text{def}} \ \{\, m \in M_1 \, \mid \, m\cdot\psi = \psi\,\} \ \ \text{is defined over $\,\Q$\,, and}
\\
&\Gamma \cap M_{1,\psi}\, \ \text{is an arithmetic subgroup of}\,\ M_{1,\psi}\,.
\end{aligned}
\end{equation}
The condition (\ref{psi1}) implies that $I_\psi F$, for $F\in C^\infty(\Gamma\backslash G)$, restricts to a
function on $M_{1,\psi}$ which is invariant under $\Gamma \cap M_{1,\psi}$\,. Thus we can define
\begin{equation}
\label{psi4}
A_\psi  :   C^\infty(\Gamma\backslash G)\ \to \ C^\infty\big((\Gamma\cap M_{1,\psi})\backslash M_{1,\psi}\big)\,, \, \ \ \ \
A_\psi F \ = \ I_\psi \big|_{M_{1,\psi}} \,.
\end{equation}
Even if $F$ is an automorphic form, $A_\psi F$ will generally not be an automorphic form on $\,M_{1,\psi}\,$, and
neither does the cuspidality of $F$ imply the cuspidality of $A_\psi F$.

The group $\,M_1\,$ inherits the property (\ref{cptctr}) of having a compact center from $\,G\,$, as was remarked
in (\ref{pone1}). However, $\,M_{1,\psi}\,$ need not satisfy this standing hypothesis on $\,G\,$, nor does
$\,M_{1,\psi}\,$ have to be semisimple. But $\,M_{1,\psi}\,$ has a Levi decomposition
\begin{equation}
\label{psi5}
\begin{aligned}
&M_{1,\psi}\ = \ U \!\cdot \! L \,,\ \ \text{with}
\\
&\qquad L\,\ \text{reductive, defined over $\Q$\,, and}
\\
&\qquad\qquad U\,\ \text{unipotent, defined over $\Q$\,, normalized by $L$.}
\end{aligned}
\end{equation}
Then $L\subset M_1$ is a reductive $\Q$-subgroup. Theorem \ref{thm:decay4} below will only use this property of
$L$, not the specific connection between $L$ and the character $\psi$ in (\ref{psi6}).

The subgroup $L$ may not be well positioned relative to $A$, i.e., $L\cap A$ need not be a maximal connected,
abelian, reductive, $\Q$-split subgroup of $L$. However, the original choices of $P$ and $A$, within the class of
minimal $\Q$-parabolic subgroups of $G$ and split components, were subject only to the conditions (\ref{pone}). We
may therefore start out with a maximal connected, abelian, reductive $\Q$-split subgroup of $L$, extend it to a
subgroup of the same type of $M_1$, and then further to a minimal $\Q$-parabolic subgroup of $M_1$. The direct
products of the resulting groups with, respectively, $A_1$ and $A_1 \! \cdot \! N_1$ can then play the roles of $A$
and $P$, without affecting the validity of (\ref{pone}). In other words, we may and shall suppose that
\begin{equation}
\label{psi6}
L \cap A\,\subset\, L\ \ \text{is maximal connected, abelian, reductive, $\Q$-split}\,.
\end{equation}
In the following, $\mathfrak S_L \subset L$ shall denote a Siegel set, or a generalized Siegel set, defined with
respect to $L \cap A$ and some minimal $\Q$-parabolic subgroup of $L$ with split component $L\cap A$, but otherwise
arbitrary.

\begin{thm}\label{thm:decay4}
Let $\tau_j$ be cuspidal. For any $c \in (M_1)_\Q$, $n\in \N$, and compact subset $S_r \subset G$, there exist
$k=k(\tau_j,n)\in\N$ and $C = C(c, \mathfrak S_L, n, k, S_r)>0$ such that for all $v\in V^\infty$,\vspace{-6pt}
\[
n_1 \in N_1\,,\ g \in \mathfrak S_L\,,\ g_r \in S_r\
\ \ \Longrightarrow\, \ \ \ |F_{\tau_j,v}(n_1 \,c\,g\,g_r)| \ \leq \ C\,\nu_k(v)\, \|g\|^{-n} \,.
\]
\end{thm}

The theorem is a fairly direct consequence of part b) of theorem \ref{thm:decay3}. We shall give the proof at the
end of this section, following some general comments and an example. In complete analogy to (\ref{siegelset2}),
there exists a finite subset $C_L \subset L_\Q$ such that, if the Siegel set $\mathfrak S_L\subset L$ is chosen
appropriately,
\begin{equation}
\label{siegelset3}
{\textstyle\bigcup}_{c\in C_L}\ \big((\Gamma \cap L)\,c\,\mathfrak S_L\big)\ = \ L\,;
\end{equation}
The estimate asserted by the theorem therefore implies the integrability of $\,h \mapsto F_{\tau_j,v}(n_1h)\,$
along the cuspidal directions in $L$, uniformly for $n_1 \in N_1$, hence also the integrability of $\,A_\psi
F_{\tau_j,v}\,$ over $(\G \cap L)\backslash L\,$; cf. (\ref{psi2}) and (\ref{psi4}). Between rapid decay and
integrability there is room to spare, so the integrability is preserved when one multiplies $\,A_\psi
F_{\tau_j,v}\,$ by a $(\G\cap L)$-invariant function of moderate growth on $L$, such as an Eisenstein series. All
this remains correct if one replaces $\,F_{\tau_j,v}\,$ with its right translate by some $g_r \in G$, and the
resulting integral is locally uniformly bounded as a function of $\,g_r\,$.

We had remarked already that $M_{1,\psi}$, and hence also $L$, need not have compact center. If not, the
integrability of $\,A_\psi F_{\tau_j,v}(h)\,$ depends not only on the rapid decay of this function as the variable
$h$ approaches a cusp, but also on the decay in the noncompact central directions. The latter is covered by the
theorem, too, of course. But it should be contrasted with the behavior of automorphic forms on a group with
noncompact center, which are generally required to transform according to a character of the
center.\stepcounter{equation}\newline

\noindent {\bf Example \theequation:} The announcement \cite{GR:1994} introduces an integral representation for the
standard $L$-function of a generic cusp form on $F_4$ involving a unipotent integration as follows. Let $G$ be the
(split) Chevalley group $F_4$, and $P_1$ be a maximal parabolic whose factor $M_1$ as in (\ref{pone}) is isomorphic
to the split $Spin(7)$. We label the nodes of the Dynkin diagram as follows:\stepcounter{equation}\vspace{-8pt}
\begin{center}\includegraphics[height=40pt]{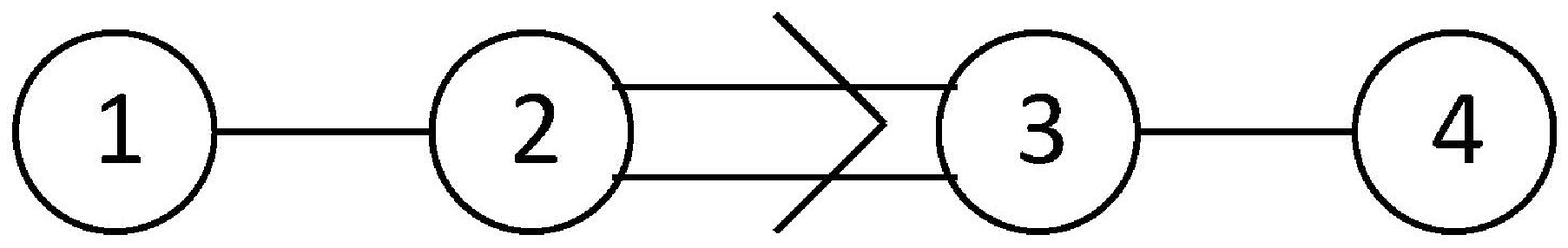}\end{center}
\vspace{-49pt}\begin{flushright}(\theequation)\end{flushright}\vspace{8pt} The reductive group $M_1$ is
characterized by the fact that its root system contains the simple roots $\a_1$, $\a_2$, and $\a_3$.   Its
unipotent radical $N_1$ is therefore the product of the fifteen root spaces corresponding to the roots which
involve the simple root $\a_4$ with a positive coefficient. This coefficient equals 1 for eight of these roots, and
equals 2 for the remaining seven. The root spaces of the latter seven roots span the Lie algebra of the center of
$N_1$, which in this case equals $[N_1,N_1]$. Hence the character variety $\widehat{N_1}\cong N_1/[N_1,N_1]$ is
eight dimensional, and the action of $M_1$ on $\widehat{N_1}$ can be described as the spin representation of
$M_1\cong Spin(7)$.  Let $\psi$, as in (\ref{psi1}), be a nontrivial character which is trivial on each of the
fifteen root spaces in ${\mathfrak n}_1$ except for the root space corresponding to the two roots $\,\alpha_1 +
\alpha_2 + \alpha_3 + \alpha_4\,$, $\,\alpha_2 + 2 \alpha_3 + \alpha_4\,$. Then the stabilizer $M_{1,\psi}$ is the
(split) Chevalley group $G_2$; in particular $M_{1,\psi}$ is reductive. Theorem~\ref{thm:decay4} asserts that
$A_{\psi}$ sends cusp forms to rapidly decaying automorphic functions on $M_{1,\psi}\cong G_2$.

Jacquet-Shalika provide an argument for a different, but related example in \cite{js2}.  In their situation, they
argue, as we did above in remark \ref{rem:decay2}, that it suffices to show rapid decay on elements in $g\in
M_{1,\psi}$ which lie in the maximal torus $M_{1,\psi}\cap A$ of $M_{1,\psi}$, and moreover on such elements in a
fixed Weyl chamber. They decompose the root spaces involved the unipotent integration over the compact fundamental
domain in (\ref{psi2}) according to whether or not conjugation by $g$ in this Weyl chamber contracts or expands the
root space.  The roots which are not expanded can be conjugated to the right into a fixed compact set and then
incorporated into $S_r$, meaning that they are harmless for decay estimates.  In the particular case considered by
Jacquet-Shalika, they observe that the remaining roots lie inside a smaller reductive subgroup, and appeal to
reduction theory to obtain an decay estimate for the integrand in (\ref{psi2}).  Though this argument works in
certain other examples, it relies on this special circumstance:~indeed, it fails in the present example because the
six expanded roots themselves span the full $F_4$ root system, and cannot fit inside a smaller reductive subgroup.
For this reason a result such as theorem \ref{thm:decay3} is needed to show the rapid
decay.\stepcounter{equation}\newline

\noindent {\it Proof\/} of theorem \ref{thm:decay4}.\ \ The factor $c$ in the argument $\,n_1\,c\,g\,g_r\,$ of
$\,F_{\tau_j,v}\,$ can be conjugated across $n_1$, all the way to the left. As in the proof of lemma
\ref{lem:phimulem} we can then omit the factor $c$ altogether, without loss of generality. Let $M_L \subset L$
denote the anisotropic factor of the centralizer of $L\cap A$ in $L$, or in other words the subgroup of $L$ which
plays the same role with respect to $L$ as $M$ does with respect to $G$. We may replace $g\in \mathfrak S_L$ in the
estimate to be proved by the product $m a\in M_L\!\cdot\!(L\cap A)$, and simultaneously replace $\|g\|^{-n}$ by
$\|a\|^{-n}$~-- this reduction of the problem is entirely analogous to deducing theorem \ref{thm:decay1} from
theorem \ref{thm:decay2}, as in remark \ref{rem:decay2}. We mentioned earlier that the group $L$ need not satisfy
the standing hypothesis (\ref{cptctr}); while the statements of theorems \ref{thm:decay1} and \ref{thm:decay1} do
depend on this hypothesis, the argument in remark \ref{rem:decay2} does not. At this point, the estimate we need
has been reduced to
\begin{equation}
\label{psi7}
\! n_1\! \in\! N_1,\, m\!\in\! M_L,\, a \!\in \!L\! \cap \! A,\, g \!\in \! S_r
 \ \Longrightarrow \  |F_{\tau_j,v}(n_1 m a g)|  \leq  C\nu_k(v) \|a\|^{-n}
\end{equation}
for some appropriately chosen $C=C(n,k,S_r)$. Since $M_L\subset M_1$ normalizes $N_1$, the factor $m\in M_L$ can be
moved across the $N_1$-factor. But $F_{\tau_j,v}$ is $\G$-invariant, and the anisotropic group $M_L$ contains
$\G\cap M_L$ as a cocompact subgroup. This allows us to restrict the factor $m\in M_L$ to a compact subset of
$M_L$, then conjugate it back across the $N_1$-factor, commute it across the factor $a\in L\cap A$, and absorb it
into the factor $g_1\in S_r$, provided $S_r$ is suitably enlarged. The resulting seemingly simpler estimate
\begin{equation}
\label{psi7} n_1 \in N_1\,,\ a \in  L\! \cap \! A\,,\ g \in  S_r \
 \ \Longrightarrow \ \ |F_{\tau_j,v}(n_1\, a \,g)| \, \leq \, C\,\nu_k(v)\, \|a\|^{-n}
\end{equation}
follows directly from part b) of theorem \ref{thm:decay3}, because $L\cap A \subset M_1\cap A$.\qed\vspace{12pt}

\section{Proof of theorem \ref{thm:decay1}}\label{sec:proofmainthm}

The assertion of the theorem for $K$-finite cuspidal automorphic forms is completely standard, of course. As was
mentioned in the introduction, theorem \ref{thm:decay1} as stated, for smooth automorphic forms, is well understood
by experts. It follows from results of Casselman and Wallach \cite{Casselman:1989,wallach}, but does not at all
require the full strength of their results. We begin by recalling the relevant details.

Let $(\pi,V)$ be a representation on  a continuous representation on a complete, locally convex topological vector
space $V$. Then
\begin{equation}
\begin{aligned}
\label{Kfinite}
V_{K-\text{finite}}\ =_{\text{def}} \ \text{span of all finite dimensional}
\\
\text{$K$-invariant subspaces is dense in $V\!$}.
\end{aligned}
\end{equation}
For this and other foundational results of Harish Chandra, Atiyah's account \cite[vol. 4, \#91]{Atiyah:collected}
is a convenient reference. For the rest of this section, we suppose more specifically that $(\pi,V)$ is a
continuous representation on a reflexive Banach space over $\C$, of {\em finite length\/}, meaning that every
nested sequence of closed, $G$-invariant subspaces terminates after a finite number of steps. In addition we
suppose that $(\pi,V)$ is {\em admissible\/}, which means
\begin{equation}
\label{admiss1}
\text{for each finite dimensional $K$-module $U$},\,\ \dim \Hom_K(U,V) \,<\, \infty\,.
\end{equation}
Irreducible unitary representations are automatically admissible.

We denote the space of $C^\infty$ vectors in $V$ by $V^\infty$, as in (\ref{Vinfinity}). By infinitesimal
translation, $\fg$ acts on $V^\infty$, which thus becomes a module over $U(\fg_\C)$, the universal enveloping
algebra of $\fg_\C = \C\otimes_\R \fg$. As consequence of the admissibility and finite length of $(\pi,V)$,
\begin{equation}
\label{admiss2}
V_{K-\text{finite}}\ \subset \ V^\infty\,.
\end{equation}
The action map $\,\fg_\C \otimes V^\infty \to V^\infty$,\,\ $X \otimes v \mapsto \pi(X)v$\,,\,\ is $K$-invariant
for purely formal reasons\begin{footnote}{$\pi(k)\left(\pi(X)v\right) = \pi(k)\pi(X)\pi(k^{-1})\pi(k)v=\pi(\Ad
k(X))\left(\pi(k)v\right)$.}\end{footnote}, and consequently $V_{K-\text{finite}}$ is $\,\fg_\C$-invariant. It
follows that $V_{K-\text{finite}}$ has the structure of $(\fg_\C,K)$-module: it is a $\,U(\fg_\C)$-module, equipped
with a locally finite\begin{footnote}{every vector lies in a finite dimensional invariant subspace.}\end{footnote}
$K$-action, which is compatible with the $\,U(\fg_\C)$-module structure in the sense that
\begin{equation}
\begin{aligned}
\label{admiss3}
\text{the infinitesimal $K$-action agrees with $\fk$ acting on $V$ via $\fk \hookrightarrow U(\fg_\C)$,}
\\
\text{and the action map $\, U(\fg_\C)\otimes V_{K-\text{finite}} \to V_{K-\text{finite}}\,$ is $K$-invariant.}
\end{aligned}
\end{equation}
Not only is every $v\in V_{K-\text{finite}}\,$ a $C^\infty$ vector, it is even an analytic\begin{footnote}{for more
elementary reasons, weakly analytic: $g\mapsto \langle \tau, \pi(g)v\rangle$ is a real analytic scalar valued
function, for every bounded linear functional $\tau : V\to \C$\,; that suffices for most applications, and in
particular suffices to establish (\ref{admiss4}).}\end{footnote} vector. This fact directly implies
\begin{equation}
\begin{aligned}
\label{admiss4}
&W \ \mapsto \ W_{K-\text{finite}}\, \ \text{establishes a bijection between closed $G$-invariant}
\\
&\qquad \text{subspaces $W \subset V$ and $(\fg_\C,K)$-submodules $\,M \subset V_{K-\text{finite}}$\,.}
\end{aligned}
\end{equation}
In particular $(\pi,V)$ is irreducible if and only if $\,V_{K-\text{finite}}\,$ is irreducible as
$(\fg_\C,K)$-module.

For $g\in G$, we let $\|g\|$ denote the matrix norm of $g$, as in (\ref{norm-1}), and $\|\pi(g)\|$ the operator
norm of $\pi(g)$ acting on the Banach space $V$. There exist positive constants $\,b,\, c\,$ such that
\begin{equation}
\label{norm}
\|\pi(g)\| \leq b\,\|g\|^c\,\ \text{for all}\,\ \,g\in G\,,
\end{equation}
as follows from the uniform boundedness principle -- see \cite[lemma 2.2]{wallach}.

It will be convenient to choose a basis $\,\{X_1,\dots,X_d\}\,$ of $\,\fg\,$. In the following $J =
(j_1,\dots,j_d)$ shall denote an ordered $d$-tuple of non-negative integers, and $|J| = j_1+ \dots + j_d\,$ the
``length" of $J$. By Poincar\'{e}-Birkhoff-Witt, the monomials
\begin{equation}
\label{infty1}
X^J \ =_{\text{def}}\ X_1^{j_1} X_2^{j_2}\dots X_d^{j_d}\,,
\end{equation}
corresponding to all possible choices of $J$, constitute a basis of the universal enveloping algebra $\,U(\fg)$.
Then
\begin{equation}
\begin{aligned}
\label{norm0}
&\text{the topology on the space of $C^\infty$ vectors $V^\infty$ is the Fr\'{e}chet}
\\
&\ \ \ \ \text{topology defined by the family of norms}\ \ \ \nu_k\, :\, V^\infty \,\rightarrow \, \R_{\geq 0}\ ,
\\
&\ \ \ \ \ \ \ \ \nu_k(v) \, = \, {\max}_{0\leq |J|\leq k}\,\|\pi(X^J)v\| \ \ \ \ \ \ (\,0 \leq k < \infty\,)\,.
\end{aligned}
\end{equation}
To see this, recall that $V^\infty$ is topologized via its inclusion into $C^\infty(G,V)$, with $v\in V^\infty$
corresponding to the $V$-valued function $g\mapsto \pi(g)v$. The topology on $C^\infty(G,V)$, in turn, is the
topology of uniform convergence on compacta of $V$-valued functions and their derivatives of all orders. Via
infinitesimal right translation, $\,U(\fg)$ is isomorphic to the algebra of all left invariant differential
operators on $G$. It follows that the topology on $V^\infty$ is defined by the family of seminorms $\,\sup_{g \in
S} \|\pi(g)\pi(Z)v\|\,$, with $S$ ranging over the compact subsets of $G$ and $Z$ over a basis of $\, U(\fg)$.
According to the uniform boundedness principle, the operator norms of $\pi(g)$, $g\in S$, are bounded on any
compact $S\subset G$. The assertion (\ref{norm0}) follows.

The density statement in (\ref{admiss2}) applies not only to $V$, but to $V^\infty$ as well. Thus $V^\infty$ can be
described also as the completion of $V_{K-\text{finite}}$ with respect to the family of norms $\nu_k$\,.

Recall the hypotheses of the current section: $(\pi,V)$ is a continuous, admissible representation of finite
length, on a reflexive Banach space. If $v\in V^\infty$ is a $C^\infty$ vector and $\tau\in (V')^{-\infty}$ --
i.e., $\tau : V^\infty \to \C$ is a continuous linear functional -- then we can define the scalar valued function
$F_{\tau,v}$ as in (\ref{autodist-1}),
\begin{equation}
\label{infty}
F_{\tau,v}(g) \ = \ \langle \,\tau \,,\, \pi(g)v \,\rangle\,.
\end{equation}
The continuity of $\tau$ means boundedness with respect to one of the seminorms $\nu_k$ -- and then, of course,
with respect to $\nu_\ell$ for any $\ell \geq k$. In this situation, there exist positive constants $b = b(\tau)$
and $c = c(\tau)$, as well as a nonnegative integer $k = k(\tau)$ such that
\begin{equation}
\label{infty4a}
\,|F_{\tau,v}(g)| \leq b\,\|g\|^c\,\nu_k(v)\,,\ \ \text{for all $v \in V^\infty$}.
\end{equation}
In effect, this is Wallach's lemma 5.1 in \cite{wallach}; it follows from (\ref{norm}) and (\ref{norm0}). One
important feature of this bound is that the order of growth does not increase if one differentiates $F_{\tau,v}$ on
the right: for $X \in \fg$, let $r(X)$ denote infinitesimal right translation by $X$; then $r(X)F_{\tau,v} =
F_{\tau,\pi(X)v}\,$, as consequence of the definition (\ref{infty}), and the exponent of growth $c$ remains
unchanged. The left invariant vector fields $r(X)$, $X\in\fg$, generate the algebra of left invariant linear
differential operators on $G$. Thus all derivatives of $F_{\tau,v}$ by such operators have the same order of
growth:
\begin{equation}
\begin{aligned}
\label{infty5a}
&\text{the functions $\,F_{\tau,v},$ with $\tau\in (V')^{-\infty}$ and $\,v \in V^\infty$,}
\\
&\qquad\ \ \ \text{have uniform moderate growth.}
\end{aligned}
\end{equation}
This applies in particular to the modular forms $F_{\tau_j,v}$ that are the subject of the earlier sections.

The usual proof of the rapid decay of $K$-finite cuspidal automorphic forms, as for example in \cite{HC1}, has two
ingredients. First, the fact that the hypothesis of moderate growth in the $K$-finite case implies uniform moderate
growth. Secondly, for a cuspidal automorphic form $F$, any fixed $g\in G$, and any subgroup $N_1 \subset N$ which
arises as the unipotent radical of a $\,\Q$-parabolic subgroup $P_1$ such that $P_1 \supset A$, the maximum of
$N_1\ni n_1\mapsto |F(n_1g)|$ can be bounded by a multiple of the maximum of $|\ell(Y_k)F(n_1g)|$, as $n_1$ ranges
over $N_1$ and $Y_k$ over a basis of the Lie algebra $\fn_1$. One may in fact suppose that $n_1$ lies in a compact
fundamental domain for the action of $\G\cap N_1$ on $N_1$, and that each $Y_k$ lies in one of the root spaces
$\fg^\alpha\subset \fn_1$; cf. (\ref{rootspace}). Then $\Ad n_1^{-1}(Y_k)$ is a bounded linear combination of the
various basis elements $Y_k$ of $\fn_1$. Let $\{X_j\}$ denote a basis of $\fg$. If $g$ lies in a Siegel set
$\mathfrak S \subset G$, and if $Y_k\in \fg^\alpha$, then $\Ad g^{-1}(Y_k)$ is a bounded linear combination of the
various $e^{-\alpha}(a)X_j$, where $a$ denotes the $A_\epsilon^+$-component of $g$ in the decomposition
(\ref{siegelset}). In this way $|F(n_1g)|$, for $n_1\in N_1$ and $g\in \mathfrak S$, can be bounded in terms of the
various $e^{-\alpha}(a)|r(X_j)F(n_1g)|$. Appropriately repeated, this argument can be used to ``push down" the
order of growth of $F$ along $\mathfrak S$ as far as one wants -- this depends on the uniform moderate growth of
$F$, of course. In view of (\ref{infty5a}), the same argument applies to the functions $F_{\tau_j,v}$ when $\tau_j$
is cuspidal, thus proving theorem \ref{thm:decay1}.

In our applications, a slight extension of theorem \ref{thm:decay1} will be useful. We suppose that $G$ can be
expressed as
\begin{equation}
\label{infty6a}
G \ = \ G_1 \!\cdot\! G_2\ \ \ \ \ \ \text{(direct product, defined over $\,\Q\,$)}
\end{equation}
of two reductive groups, both of which satisfy the original hypotheses on $G$. We consider an automorphic
distribution $\tau_j \in \big((V')^{-\infty}\big)^\G$ as in section \ref{sec:rapiddecay}, but require it only to be
cuspidal with respect to the factor $G_1$, in the sense that the cuspidality condition (\ref{cuspidal2}) holds for
all subgroups $U\subset G_1$ which arise as the unipotent radical of a proper $\,\Q$-parabolic subgroup of $G_1$.
We select norms $\,g_i \mapsto \|g_i \|_i\,$ on $\,G_i\,$, $\,i=1,2\,$, as in (\ref{norm-1}).

\begin{thm}
\label{thm:decay5} Let $\mathfrak S_1$ be a generalized Siegel set in $G_1$. If $\tau_j$ is cuspidal with respect
to the factor $G_1$, there exists an integer $N_2=N_2(\tau_j)$, not depending on $\,\mathfrak S_1$, with the
following property. For each $n\in \N$ one can choose a positive integer $k=k(\tau_j,n)$ and a positive constant
$C_1=C_1(\tau_j,\mathfrak S_1,n,k)$ such that\vspace{-2\jot}
\[
g_1 \!\in\! \mathfrak S_1\,,\ g_2 \!\in\! G_2\,,\ v \!\in\! V^\infty  \ \ \Longrightarrow \ \ |F_{\tau_j,v}(g_1 \!\cdot\! g_2)| \, \leq \, C_1\,\nu_k(v)\|g_1\|^{-n}\|g_2\|^{N_2} \,.
\]
\end{thm}

Roughly paraphrased, $F_{\tau_j,v}(g_1 \!\cdot\! g_2)$ simultaneously exhibits rapid decay in the variable $g_1 \in
\mathfrak S_1$ and moderate growth in the variable $g_2\in G_2$. Theorems \ref{thm:decay2}, \ref{thm:decay3}, and
\ref{thm:decay4}, all of which were deduced from theorem \ref{thm:decay1}, have analogous extensions, of course.
The proof of theorem \ref{thm:decay5} proceeds along the same line as that of theorem \ref{thm:decay1}, except that
the ``pushing down" of exponents is performed -- and can only be performed -- in the first factor $G_1$.

\section{Uniform moderate growth}
\label{sec:uniform}

In this section, we shall prove theorem \ref{thm:uniformlymg}, using a result of Averbuch \cite{Averbuch:1986}. We
consider a smooth $\Gamma$-automorphic form $F$, of moderate growth, but do not explicitly require $F$ to have
uniform moderate growth. In other words, $\,F \in C^\infty(\G\backslash G)\,$ satisfies the two conditions
\begin{equation}
\label{uniform1}
\begin{alignedat}{2}
&\text{a)}\ \ &&|F(g)| \ \leq \ C\,\|g\|^N\ \ \text{for some $\,C>0\,$ and $\,N \in \N\,$,\,\ and}
\\
&\text{b)}\ \ &&F\,\ \text{transforms finitely under the action of the center of $\, U(\fg_\C)\,$}.
\end{alignedat}
\end{equation}
Averbuch's result asserts that if $F$ is transforms finitely under the Hecke action, the exponent of growth $N$ in
(\ref{uniform1}a) can be chosen solely in terms of the eigenvalues and their multiplicities of the Hecke action on
the finite dimensional space generated by the Hecke translates of $F$, but otherwise independently of $F$. In
particular,

\begin{prop}\label{prop:uniform}
The order of growth of $F$ is uniform if $F$ transforms finitely under the Hecke action.
\end{prop}

Indeed, the Hecke algebra acts on the left, and therefore commutes with the action of any left invariant
differential operator. According to Averbuch, applying such a differential operator does not affect the order of
growth.

Let us suppose now that $F\in C^\infty(\G\backslash G)$ satisfies (\ref{uniform1}), and in addition the cuspidality
condition
\begin{equation}
\label{uniform2}
\begin{aligned}
\int_{(\G\cap U)\backslash U} f(u\,g)\,du \, = \, 0\, \ \ \text{for all $g\in G$ and any $\,U\subset G$ that arises}
\\
\text{as the unipotent radical of a proper $\Q$-parabolic $P\subset G$}.
\end{aligned}
\end{equation}
Because of proposition \ref{prop:uniform}, the first assertion of theorem \ref{thm:uniformlymg} reduces to:

\begin{lem}
\label{lem:uniform} Under the hypotheses just stated, $F$ is Hecke finite.
\end{lem}

\begin{proof}
We choose $N\in \N$ as in (\ref{uniform1}a). Since $\G\backslash G$ has finitely many cusps, there exists a smooth,
strictly positive, $\G$-invariant measure $dm$ on $G$ that is bounded by a multiple of $\,\|g\|^{-2N}dg\,$ near the
``end" of each cusp; here $dg$ refers to Haar measure, as usual. We let $\,\widetilde W^\infty\,$ denote the vector
space of all functions $f\in C^\infty(\G\backslash G)$ which are square integrable with respect to $dm$, and are
cuspidal in the sense of (\ref{uniform2}). Then $\,\widetilde W$, the completion of $\,\widetilde W^\infty$ in the
$L^2(dm)$-norm, is a Hilbert space on which $\,G\,$ acts by right translation, in a continuous manner. It contains
$F$, and
\begin{equation}
\label{app4}
W \ = \ \text{closure in $\,\widetilde W\,$ of}\, \ \{\,r(g)F\,\mid \, g\in G\,\}
\end{equation}
is a $G$-invariant subspace of $\,\widetilde W^\infty$; here $r(g)$ denotes right translation by $g$.

All the $\,f \in W\,$ are locally $L^2$ as functions on $G$, and thus can be regarded as distributions. They
inherit the finiteness condition (\ref{uniform1}b) from $\,F\,$:
\begin{equation}
\label{app5}
\begin{aligned}
&\text{there exists an ideal $\,I_F\,$ of finite codimension in the center}
\\
&\text{of $\, U(\fg_\C)$, such that $\,r(Z)f=0\,$ for every $\,f\in W\,$ and $\,Z\in I_F$\,};
\end{aligned}
\end{equation}
the equality $\,r(Z)f=0\,$ is to be interpreted in the sense of distributions, of course. The ideal $\,I_F\,$
contains monic polynomials in the Casimir operator $\,\Omega\,$, and $\,\Omega\,$ is elliptic transversally to
$\,K\,$. Thus any $\,K$-finite $\,f \in W\,$ satisfies an elliptic differential equation, and therefore
\begin{equation}
\label{app6}
W_{K-\text{finite}} \ \subset \ W \cap C^\infty(\G\backslash G)\,.
\end{equation}
Let $\widehat K$ denote the set of isomorphism classes of irreducible, finite dimensional representations of $K$,
and $W_i$, for $i\in \widehat K$, the space of $i$-isotypic vectors~-- i.e., functions $f\in W$ which transform
according to the class $i$ under $K$. Then
\begin{equation}
\label{app6a}
W_{K-\text{finite}} \ = \ {\oplus}_{i\in\widehat K}\ W_i\,.
\end{equation}
We shall show that the $W_i$ are finite dimensional.

Each $W_i$ is the image of a $K$-invariant projection operator $p_i : W \to W_i\,$, given by right convolution,
under $K$, against the complex conjugate of the character of $i$, suitably renormalized. In particular $p_i$ is
bounded. Like $F$, the right translates $r(g)F$, $g\in G$, satisfy (\ref{uniform1}a). Thus, in view of
(\ref{app4}),
\begin{equation}
\label{app6b}
W_i^N \ =_{\text{def}} \ \{\,f\in W_i \, \mid\, {\sup}_{g\in G}\,\|g\|^{-N}|f(g)|<\infty\,\}\ \ \text{is dense in}\,\ W_i\,.
\end{equation}
Both $p_i$ and right translation by any $g\in G$ preserve the cuspidality condition (\ref{uniform2}). Thus each
$f\in W_i$ is $K$-finite, cuspidal, of moderate growth, and $Z(\fg)$-finite; in other words,
\begin{equation}
\label{app6c}
W_i^N \ \ \text{consists of $K$-finite cuspidal automorphic forms}\,.
\end{equation}
The cuspidal spectrum in $\,L^2(\G\backslash G)\,$ breaks up discretely, with finite multiplicities \cite{GPS}.
Only finitely many irreducible Harish Chandra modules are annihilated by the ideal $\,I_F$, and every irreducible
Harish Chandra module has finite $\,K$-multiplicities. It follows that $W_i^N$ is finite dimensional, hence
\begin{equation}
\label{app6d}
\dim W_i \ <\ \infty \,,
\end{equation}
as consequence of (\ref{app6b}). Moreover, all the $W_i$, $i\in \widehat K$, lie in a finite direct sum of cuspidal
automorphic representations.

Recall the definition (\ref{admiss1}) of an admissible representation. Because of what we have just established,
the $(\fg_\C,K)$-module $W_{K-\text{finite}}$ has finite $K$-multiplicities, so $(\,r,W)$ is indeed admissible. But
$W_{K-\text{finite}}$ is also completely reducible, since it lies in a finite direct sum of $G$-irreducible
subspaces of $L^2(\G\backslash G)$. In view of (\ref{Kfinite}) and (\ref{admiss4}), $(\,r,W)$ must therefore have
finite length and be completely reducible. To summarize,
\begin{equation}
\label{uniform3}
\begin{gathered}
\text{the representation $\,(\,r,W)\,$ is admissible and completely}
\\
\text{reducible, with finitely many irreducible constituents}.
\end{gathered}
\end{equation}
Appealing once more to (\ref{Kfinite}), we now find that $F$ lies in the $L^2(dm)$-closure of
$W_{K-\text{finite}}$, and therefore in the $L^2(dm)$-closure of the linear span of the finitely many irreducible
summands of the cuspidal spectrum which are annihilated by $I_F$. The Hecke action commutes with the action of
$I_F$, and hence the summands can be chosen so that the Hecke algebra acts on each of them by a character. The
assertion of the lemma now follows, because the Hecke action is continuous with respect to the $L^2(dm)$-norm.
\end{proof}

To complete the proof of theorem \ref{thm:uniformlymg}, we must show that any $F$ satisfying the hypotheses stated
at the beginning of this section is a sum of $C^\infty$ vectors in finitely many closed $G$-invariant,
$G$-irreducible subspaces of $L^2(\G\backslash G)$~-- see remark \ref{ftaujv}.

At this point, as consequence of proposition \ref{prop:uniform} and lemma \ref{lem:uniform}, we know that $F$ has
uniform moderate growth. Thus we can apply the usual proof of the rapid decay of cusp forms, as sketched in section
\ref{sec:proofmainthm}, to conclude that $F$ decays rapidly on Siegel sets; not only $F$, in fact, but also all its
derivatives $r(X)F$ with $X\in U(\fg_\C)$, since these satisfy the same hypotheses as $F$. In particular,
\begin{equation}
\label{app6e}
r(X)F \in L^2(\G\backslash G)\,,\ \ \ \text{for all}\,\ X\in U(\fg_\C)\,.
\end{equation}
This means that $g\mapsto r(g)F$ is a $C^\infty$ function on $G$, with values in the Hilbert space
$L^2(\G\backslash G)$.

We can now argue as in the proof of lemma \ref{lem:uniform}, using Haar measure $dg$ instead of $dm$. The space
$W$, which contains $F$, is then a closed subspace of $L^2(\G\backslash G)$. The assertion (\ref{uniform3}), in the
present context, means that $W= \oplus_{1\leq k\leq m}\,W_k$ is a finite direct sum of closed $G$-invariant,
$G$-irreducible subspaces $W_k \subset L^2(\G\backslash G)$. The orthogonal projections $L^2(\G\backslash G)\to
W_k$ are bounded linear maps, and composing the $C^\infty$ function $g\mapsto r(g)F$ with these projections results
in $C^\infty$ $W_k$-valued functions~-- in words, in $C^\infty$ vectors $F_k\in W_k^\infty$. But then $F=
\sum_{1\leq k \leq m} F_k\,$, which is the conclusion we need.

\appendix

\section{Appendix}\label{appendix}

The integral (\ref{cuspidal2}) makes sense because $V^{-\infty}$ is a complete, locally convex, Hausdorff
topological vector space, on which $G$ acts in a strongly continuous manner. These facts are well known to experts,
but it seems no concise reference exists. We shall briefly sketch the relevant arguments in this appendix.

Recall the notion of {\it continuity} of a representation $(\pi,V)$ on a complete, locally convex, Hausdorff
topological vector space: the action map from $G \times V$ to $V$, $(g,v)\mapsto \pi(g)v$, is continuous, relative
to the product topology on the domain. {\it Strong continuity}, on the other hand, only requires the continuity of
the functions $g \mapsto \pi(g)v$ for all $v \in V$. The terminology notwithstanding, ``continuity" thus is a more
restrictive notion than ``strong continuity". The uniform boundedness principle ensures that the two definitions
coincide for Banach spaces. If $V$ is a Banach space, one can argue directly from the definition that the action of
$G$ on $C^\infty(G,V)$, the space of smooth $V$-valued functions, is continuous. Via the assignment $\,v \mapsto
F_v\,$, $\,F_v(g)=\pi(g)v\,$, $V^\infty$ maps into $C^\infty(G,V)$. Thus, if $(\pi,V)$ is a continuous
representation on a Banach space, the representation $(\pi,V^\infty)$ on the space of $C^\infty$ vectors inherits
the continuity from $C^\infty(G,V)$, which contains $V^\infty$ as a closed subspace.

Still in the setting of a Banach representation $(\pi,V)$, the continuity of the induced representation
$(\pi,V^{-\infty})$ on the space of distribution vectors is a more subtle matter. By definition, $V^{-\infty}$ is
the space of continuous linear functionals on $(V')^\infty$, the space of $C^\infty$ vectors for the dual
representation $(\pi',V')$ on the Banach dual $V'$. But the continuity of $\pi$ implies the strong continuity of
the dual representation $\pi'$ only if Banach space $V$ is reflexive \cite[vol. I, Proposition
4.2.2.1]{Warner:1972}. Only in that case is the continuity of $(\pi',(V')^\infty)$ assured, and therefore only in
that case does it make sense to consider the representation $(\pi,V^{-\infty})$ on the space of distribution
vectors $V^{-\infty}$.

Let us suppose then that $V$ is a reflexive Banach space, or more specifically a Hilbert space as in
(\ref{autorep}). We apply (\ref{norm0}) to $V'$ and let $(V')^k$ denote the completion of $V'_{K-\text{finite}}$
with respect to the norm $\nu_k$. In ana\-logy to $V^\infty$,\,\ $(V')^k$ can be topologically embedded into
$C^k(G,V')$ as a closed subspace. That implies the continuity of $(\pi',(V')^k)$, which alternatively can be
deduced directly from (\ref{norm0}). These arguments also topologically identify $(V')^\infty$ with the projective
limit \cite[Chap. 50-7]{treves:1967} of the $(\pi',(V')^k)$,
\begin{equation}
\label{infty4b}
(V')^\infty \ \cong \ \lim_{\leftarrow}\, (V')^k\,.
\end{equation}
Dually, we equip $V^{-\infty}$ with the inductive limit topology
\begin{equation}
\label{infty5b}
V^{-\infty} \ \cong \ \lim_{\rightarrow}\, V^{-k}\,,\ \ \ \ V^{-k} \ =_{\text{def}} \ \left((V')^k\right)'\,,
\end{equation}
for the sequence of Banach duals $V^{-k}$, i.e., the strongest locally convex topo\-logy that makes the inclusions
$V^{-k} \hookrightarrow V^{-\infty}$ continuous. These inclusions are then topological isomorphisms onto their
images, and $V^{-\infty}$ is the union of the $V^{-k}$; see, for example, Chapters 13 and 50-7 in
\cite{treves:1967}. Thus, to prove the strong continuity of $(\pi,V^{-\infty})$, it suffices to establish the
strong continuity of the $(\pi,V^{-k})$. By definition, each $(V')^k$ is isomorphic as Banach space -- though not
as $G$-representation -- to the closure of the image of the map
\begin{equation}
\label{infty6b}
V'_{K-\text{finite}} \ \rightarrow \ V' \oplus V' \oplus \dots \oplus V'\,,\ \ \ \ v' \ \mapsto \ (X^Jv')_{0 \leq |J| \leq k}\ ,
\end{equation}
hence is isomorphic to a closed subspace of a finite direct sum of reflexive Banach spaces, hence is itself a
reflexive Banach space. Thus the strong continuity of the $(\pi',(V')^k)$ implies the strong continuity of the
$(\pi,V^{-k})$, and hence the strong continuity of $(\pi,V^{-\infty})$. With more effort one can prove that
$(\pi,V^{-\infty})$ is even continuous, but strong continuity is enough for our purposes since it gives meaning to
the integral in (\ref{cuspidal2}).

\vspace{1cm}
\begin{tabular}{lcl}
Stephen D. Miller                    & & Wilfried Schmid \\
Department of Mathematics         & & Department of Mathematics \\
Hill Center-Busch Campus          & & Harvard University \\
Rutgers, The State University of New Jersey             & & Cambridge, MA 02138 \\
 110 Frelinghuysen Rd             & & {\tt schmid@math.harvard.edu}\\
 Piscataway, NJ 08854-8019\\
 {\tt miller@math.rutgers.edu}  \\

\end{tabular}

\end{document}